\RequirePackage[l2tabu,orthodox]{nag}		
\documentclass[reqno]{amsart}		
\usepackage[margin=1.5in,bottom=1.25in]{geometry}		

\usepackage{amsmath}		
\usepackage{amssymb}		
\usepackage{amsfonts}		
\usepackage{amsthm}		
\usepackage[foot]{amsaddr}		

\usepackage{mathtools}		
\mathtoolsset{%
}

\newtagform{comment}{\{}{\}}

\usepackage[utf8]{inputenc}		
\usepackage[T1]{fontenc}		

\usepackage[proportional,tabular,lining,sf,mono=false]{libertine}

\usepackage{dsfont}		

\usepackage[
cal=cm,
]
{mathalfa}

\usepackage{acronym}		
\renewcommand*{\aclabelfont}[1]{\acsfont{#1}}		
\newcommand{\acli}[1]{\textit{\acl{#1}}}		
\newcommand{\acdef}[1]{\textit{\acl{#1}} \textup{(\acs{#1})}\acused{#1}}		

\usepackage[labelfont={bf,small},labelsep=colon,font=small]{caption}	

\usepackage[dvipsnames,svgnames]{xcolor}		
\colorlet{MyRed}{Crimson!70!Black}
\colorlet{MyBlue}{MediumBlue!90!black}
\colorlet{MyGreen}{DarkGreen!60!black}

\newcommand{\afterhead}{}
\newcommand{\ackperiod}{}		
\newcommand{\para}[1]{\medskip\paragraph{\textbf{#1\afterhead}}}

\usepackage{wasysym}		

\usepackage{subcaption}		
\usepackage{wrapfig}		

\usepackage{tikz}		
\usetikzlibrary{calc,patterns}

\usepackage{array}		
\usepackage{booktabs}		
\usepackage[inline,shortlabels]{enumitem}		
\setenumerate{itemsep=\smallskipamount,topsep=\smallskipamount}
\setitemize{itemsep=\smallskipamount,topsep=\smallskipamount}

\usepackage[kerning=true]{microtype}		

\usepackage{latexsym}		
\usepackage{xspace}		

\usepackage[numbers,sort&compress]{natbib}		

\usepackage{hyperref}
\hypersetup{
colorlinks=true,
linktocpage=true,
pdfstartview=FitH,
breaklinks=true,
pdfpagemode=UseNone,
pageanchor=true,
pdfpagemode=UseOutlines,
plainpages=false,
bookmarksnumbered,
bookmarksopen=false,
bookmarksopenlevel=1,
hypertexnames=true,
pdfhighlight=/O,
urlcolor=MyBlue,linkcolor=MyBlue,citecolor=MyBlue,	
pdftitle={},
pdfauthor={},
pdfsubject={},
pdfkeywords={},
pdfcreator={pdfLaTeX},
pdfproducer={LaTeX with hyperref}
}

\usepackage[sort&compress,capitalize,nameinlink]{cleveref}		
\crefname{assumption}{Assumption}{Assumptions}

\usepackage[textwidth=30mm]{todonotes}		
\newcommand{\debug}[1]{#1}		

\usepackage{algorithm}		
\usepackage{algpseudocode}		

\usepackage{thmtools}		
\usepackage{thm-restate}		

\theoremstyle{plain}
\newtheorem{theorem}{Theorem}		
\newtheorem{corollary}{Corollary}		
\newtheorem{lemma}{Lemma}		
\newtheorem{proposition}{Proposition}		

\newtheorem*{theorem*}{Theorem}		
\newtheorem*{corollary*}{Corollary}		

\theoremstyle{definition}
\newtheorem{definition}{Definition}		
\newtheorem{assumption}{Assumption}		

\newtheorem*{definition*}{Definition}		
\newtheorem*{assumption*}{Assumptions}		

\theoremstyle{remark}
\newtheorem{remark}{Remark}		

\newtheorem*{remark*}{Remark}		
\newtheorem*{example*}{Example}		

\newcounter{proofpart}

\newcommand{\newmacro}[2]{\newcommand{#1}{\debug{#2}}}		

\DeclarePairedDelimiter{\braces}{\{}{\}}		
\DeclarePairedDelimiter{\bracks}{[}{]}		
\DeclarePairedDelimiter{\parens}{(}{)}		

\DeclarePairedDelimiter{\abs}{\lvert}{\rvert}		

\DeclarePairedDelimiterX{\setdef}[2]{\{}{\}}{#1:#2}		
\DeclarePairedDelimiterXPP{\exclude}[1]{\mathopen{}\setminus}{\{}{\}}{}{#1}

\newcommand{\cf}{cf.\xspace}		
\newcommand{\eg}{e.g.,\xspace}		
\newcommand{\ie}{i.e.,\xspace}		
\newcommand{\vs}{vs.\xspace}		

\newcommand{\textpar}[1]{\textup(#1\textup)}		

\newcommand{\txs}{\textstyle}		

\newcommand{\alt}[1]{#1'}		

\newcommand{\R}{\mathbb{R}}		

\DeclareMathOperator{\bigoh}{\mathcal O}		
\DeclareMathOperator{\Grass}{\mathbf{Gr}}		
\DeclareMathOperator{\Jac}{Jac}		

\newmacro{\coef}{\lambda}		
\newmacro{\dd}{\:d}		
\newmacro{\subs}{\leftarrow}		

\newcommand{\ddt}{\frac{d}{dt}}		
\newcommand{\eps}{\varepsilon}		
\newcommand{\pd}{\partial}		

\DeclareMathOperator*{\intersect}{\bigcap}		
\DeclareMathOperator*{\union}{\bigcup}		

\DeclareMathOperator{\im}{im}		
\DeclareMathOperator{\one}{\mathds{1}}		

\newcommand{\from}{\colon}		
\newcommand{\injects}{\hookrightarrow}		

\newmacro{\points}{\R^{\vdim}}		
\newmacro{\intpoints}{\points^{\circ}}		
\newmacro{\point}{x}		
\newmacro{\pointalt}{\alt\point}		

\newmacro{\dpoints}{\mathcal{Y}}		
\newmacro{\dpoint}{y}		
\newmacro{\dpointalt}{\alt\dpoint}		

\newmacro{\base}{p}		
\newmacro{\basealt}{q}		

\DeclareMathOperator{\dist}{dist}		

\newmacro{\open}{\mathcal{U}}		
\newmacro{\closed}{\mathcal{C}}		
\newmacro{\cpt}{\mathcal{K}}		
\newmacro{\nhd}{\mathcal{U}}		

\newmacro{\start}{1}		
\newmacro{\afterstart}{2}		
\newmacro{\running}{\start,\afterstart,\dotsc}		

\newmacro{\runalt}{k}		
\newmacro{\run}{n}		
\newmacro{\nRuns}{T}		
\newmacro{\runs}{\mathcal{\nRuns}}		

\newmacro{\state}{X}		
\newmacro{\dstate}{Y}		
\newmacro{\aux}{Z}		

\newcommand{\init}[1][\state]{\debug{#1}_{\start}}		
\newcommand{\preiter}[1][\state]{\debug{#1}_{\runalt-1}}		
\newcommand{\iter}[1][\state]{\debug{#1}_{\runalt}}		
\newcommand{\preprev}[1][\state]{\debug{#1}_{\run-2}}		
\newcommand{\prev}[1][\state]{\debug{#1}_{\run-1}}		
\newcommand{\curr}[1][\state]{\debug{#1}_{\run}}		
\renewcommand{\next}[1][\state]{\debug{#1}_{\run+1}}		

\newmacro{\ctime}{t}		
\newmacro{\ctimealt}{s}		
\newmacro{\cstart}{0}		

\newcommand{\orbit}[2][]{\point_{#1}(#2)}		
\newcommand{\dorbit}[2][]{\dot\point_{#1}(#2)}		

\newmacro{\horizon}{T}		

\newmacro{\vecspace}{\points}		
\newmacro{\vdim}{d}		
\newmacro{\vvec}{v}		
\newmacro{\bvec}{e}		
\newmacro{\unitvec}{u}		

\newmacro{\subspace}{\mathcal{W}}		
\newmacro{\wpoint}{y}		
\newmacro{\wvec}{w}		

\newmacro{\thull}{\mathcal{Z}}		
\newmacro{\tvec}{z}		

\DeclarePairedDelimiterX{\braket}[2]{\langle}{\rangle}{#1,#2}		
\newcommand{\dual}[1]{#1^{\ast}}		

\newmacro{\dspace}{\dual\vecspace}		
\newmacro{\dvec}{v}		
\newmacro{\dbvec}{\eps}		

\newmacro{\ones}{\mathbf{1}}		
\newmacro{\mat}{M}		
\newmacro{\eye}{I}		

\newcommand{\mg}{\succ}		
\newcommand{\mleq}{\preccurlyeq}		

\newmacro{\cvx}{\mathcal{C}}		
\newmacro{\subd}{\partial}		
\newmacro{\hmat}{H}		

\newmacro{\strong}{\alpha}		
\newmacro{\smooth}{L}		
\newmacro{\lips}{G}		

\newmacro{\hreg}{h}		
\newmacro{\breg}{D}		
\newmacro{\proxmap}{P}		
\newmacro{\mirror}{Q}		
\newmacro{\fench}{F}		
\newmacro{\hstr}{K}		
\newmacro{\depth}{H}		
\newmacro{\zone}{\mathbb{D}}		
\newmacro{\subpoints}{\points^{\circ}}		

\DeclarePairedDelimiterXPP{\prox}[2]{\proxmap_{#1}}{(}{)}{}{#2}		

\DeclareMathOperator*{\argmin}{arg\,min}		
\DeclareMathOperator{\crit}{crit}		

\newmacro{\obj}{f}		
\newmacro{\objalt}{g}		
\newmacro{\sobj}{F}		

\newmacro{\param}{\theta}		
\newmacro{\params}{\Theta}		

\newmacro{\gvec}{g}		
\newmacro{\vecfield}{v}		
\newmacro{\oper}{A}		

\newmacro{\gbound}{G}		
\newmacro{\vbound}{G}		

\newcommand{\sol}[1][\point]{#1^{\ast}}		
\newmacro{\sols}{\sol[\mathcal{X}]}		

\newmacro{\minmax}{\Phi}		

\newmacro{\minvar}{\point_{1}}		
\newmacro{\minvaralt}{\alt\minvar}		
\newmacro{\minvars}{\points_{1}}		

\newmacro{\maxvar}{\point_{2}}		
\newmacro{\maxvars}{\points_{2}}		
\newmacro{\maxvaralt}{\alt\maxvar}		

\newmacro{\play}{i}		
\newmacro{\playalt}{j}		
\newmacro{\nPlayers}{N}		
\newmacro{\players}{\mathcal{\nPlayers}}		

\newmacro{\pure}{a}		
\newmacro{\purealt}{a'}		
\newmacro{\nPures}{A}		
\newmacro{\pures}{\mathcal{\nPures}}		

\newmacro{\cost}{c}		
\newmacro{\loss}{\ell}		
\newmacro{\pay}{u}		
\newmacro{\payv}{v}		
\newmacro{\pot}{\obj}		

\newmacro{\game}{\mathcal{G}}		
\newmacro{\gamefull}{\game(\players,\points,\pay)}		

\newmacro{\fingame}{\Gamma}		
\newmacro{\fingamefull}{\Gamma(\players,\pures,\pay)}		

\DeclareMathOperator{\ex}{\mathbb{E}}		
\DeclareMathOperator{\prob}{\mathbb{P}}		

\newmacro{\sample}{\omega}		
\newmacro{\samples}{\Omega}		
\newmacro{\filter}{\mathcal{F}}		
\newmacro{\probspace}{(\samples,\filter,\prob)}		

\newmacro{\event}{\samples}		
\newmacro{\eventalt}{E}		
\newcommand{\comp}[1]{#1^{\mathtt{c}}}		

\newmacro{\mean}{\mu}		
\newmacro{\sdev}{\sigma}		
\newmacro{\variance}{\sdev^{2}}		

\newmacro{\dkl}{D_{\mathrm{KL}}}		
\newcommand{\as}{\debug{\textpar{a.s.}}\xspace}		

\providecommand\given{}		

\DeclarePairedDelimiterXPP{\exof}[1]{\ex}{[}{]}{}{
\renewcommand\given{\nonscript\,\delimsize\vert\nonscript\,\mathopen{}} #1}

\DeclarePairedDelimiterXPP{\probof}[1]{\prob}{(}{)}{}{
\renewcommand\given{\nonscript\:\delimsize\vert\nonscript\:\mathopen{}} #1}

\newcommand{\oneof}[1]{\one_{\braces*{#1}}}

\newmacro{\step}{\gamma}		
\newmacro{\temp}{\eta}		

\newmacro{\proper}{\tau}		
\newcommand{\apt}[2][]{\state_{#1}(#2)}		

\newmacro{\orcl}{\mathsf{V}}		
\newmacro{\signal}{V}		

\newmacro{\error}{Z}		
\newmacro{\noise}{Z}		
\newmacro{\bias}{b}		
\newmacro{\brown}{W}		

\newmacro{\snoise}{\xi}		
\newmacro{\sbias}{\psi}		
\newmacro{\scorr}{\theta}		

\newmacro{\sbound}{M}		
\newmacro{\bbound}{B}		
\newmacro{\noisepar}{\sdev}		
\newmacro{\noisevar}{\variance}		

\newmacro{\mix}{\delta}		
\newmacro{\pexp}{p}		
\newmacro{\qexp}{q}		
\newmacro{\rexp}{r}		

\newmacro{\pert}{w}		
\newmacro{\pertvar}{W}		
\newmacro{\unitvar}{\unitvec}		

\newcommand{\olim}[1][\point]{\hat#1}		

\newmacro{\flowmap}{\Phi}		
\newcommand{\flow}[2]{\flowmap_{#1}(#2)}		
\newmacro{\genmap}{\Psi}		
\newcommand{\gen}[2]{\genmap_{#1}(#2)}		

\newcommand{\saddle}{\sol}		
\newmacro{\saddles}{\mathcal{S}}		

\newmacro{\graph}{\mathcal{G}}
\newmacro{\vertices}{\mathcal{V}}
\newmacro{\edges}{\mathcal{E}}

\DeclarePairedDelimiterX{\product}[2]{\langle}{\rangle}{#1,#2}		
\DeclarePairedDelimiter{\norm}{\lVert}{\rVert}		
\newcommand{\dnorm}[1]{\norm{#1}}		

\newmacro{\gmat}{g}		
\newmacro{\gdist}{\dist_{\gmat}}

\newmacro{\ball}{\mathbb{B}}		
\newmacro{\sphere}{\mathbb{S}}		
\newmacro{\radius}{r}		

\newmacro{\mfld}{\mathcal{M}}		
\newcommand{\tangent}[2]{\debug{T}_{#1}{#2}}		

\newmacro{\const}{c}
\newmacro{\offset}{m}

\newmacro{\pcoef}{P}		
\newmacro{\rcoef}{R}		

\newmacro{\qmin}{\alpha}		
\newmacro{\qmaj}{\beta}		

\newmacro{\rconst}{R_{\ast}}

\newmacro{\flvl}{L}
\newmacro{\glvl}{M}
\newmacro{\critval}{a}
\newmacro{\critnhd}{K}

\newmacro{\prelyap}{V}		
\newmacro{\lyap}{E}		

\newcommand{\zerdirs}[1]{\debug{\mathcal{E}}_{#1}^{c}}		
\newcommand{\posdirs}[1]{\debug{\mathcal{E}}_{#1}^{s}}		
\newcommand{\negdirs}[1]{\debug{\mathcal{E}}_{#1}^{u}}		

\begin{document}


\newcommand{\longtitle}{\uppercase{On the Almost Sure Convergence of\\
Stochastic Gradient Descent in Non-Convex Problems}}
\newcommand{\runtitle}{\uppercase{Stochastic Gradient Descent in Non-Convex Problems}}

\title
[\runtitle]
{\longtitle}		

\author
[P.~Mertikopoulos]
{Panayotis Mertikopoulos$^{\ast,\diamond,\lowercase{c}}$}
\email{panayotis.mertikopoulos@imag.fr}

\author
[N.~Hallak]
{Nadav Hallak$^{\sharp}$}
\email{nadav.hallak@epfl.ch}

\author
[A.~Kavis]
{Ali Kavis$^{\sharp}$}
\email{ali.kavis@epfl.ch}

\author
[V.~Cevher]
{Volkan Cevher$^{\ast}$}
\email{put.email@here}

\address{$^{\ast}$\,%
Univ. Grenoble Alpes, CNRS, Inria, LIG, 38000, Grenoble, France.}
\address{$^{\diamond}$\,%
Criteo AI Lab.}
\address{$^{\sharp}$\,%
École Polytechnique Fédérale de Lausanne (EPFL).}
\address{$^{c}$\,Corresponding author.}



\thanks{
%
%
This research was partially supported by the COST Action CA16228 ``European Network for Game Theory'' (GAMENET). P.~Mertikopoulos is also grateful for financial support by
the French National Research Agency (ANR) under grant no.~ANR\textendash 16\textendash CE33\textendash 0004\textendash 01 (ORACLESS).
N.~Hallak and A.~Kavis were supported by the European Research Council (ERC) under the European Union's Horizon 2020 research and innovation programme (grant agreement no 725594 - time-data).
V.~Cevher gratefully acknowledges the support of the Swiss National Science Foundation (SNSF) under grant \textnumero\ 200021\textendash 178865 / 1, the European Research Council (ERC) under the European Union's Horizon 2020 research and innovation programme (grant agreement \textnumero\ 725594 - time-data), and 2019 Google Faculty Research Award\ackperiod}

\newacro{LHS}{left-hand side}
\newacro{RHS}{right-hand side}
\newacro{iid}[i.i.d.]{independent and identically distributed}
\newacro{lsc}[l.s.c.]{lower semi-continuous}

\newacro{APT}{asymptotic pseudotrajectory}
\newacroplural{APT}[APTs]{asymptotic pseudotrajectories}
\newacro{GD}{gradient dynamics}
\newacro{GF}{gradient flow}
\newacro{MDS}{martingale difference sequence}
\newacro{NHIM}{normally hyperbolic invariant manifold}
\newacro{NHSM}{normally hyperbolic saddle manifold}
\newacro{NHM}{normally hyperbolic manifold}
\newacro{ODE}{ordinary differential equation}
\newacro{RM}{Robbins\textendash Monro}
\newacro{SA}{stochastic approximation}
\newacro{SG}{stochastic gradient}
\newacro{SGD}{stochastic gradient descent}
\newacro{SFO}{stochastic first-order oracle}

\begin{abstract}
%
%
This paper analyzes the trajectories of \ac{SGD} to help understand the algorithm's convergence properties in non-convex problems.
We first show that the sequence of iterates generated by \ac{SGD} remains bounded and converges with probability $1$ under a very broad range of step-size schedules.
Subsequently, going beyond existing positive probability guarantees, we show that \ac{SGD} avoids strict saddle points/manifolds with probability $1$ for the entire spectrum of step-size policies considered.
Finally, we prove that the algorithm's rate of convergence to Hurwicz minimizers is $\bigoh(1/\run^{\pexp})$ if the method is employed with a $\Theta(1/\run^{\pexp})$ step-size.
This provides an important guideline for tuning the algorithm's step-size as it suggests that a cool-down phase with a vanishing step-size could lead to faster convergence;
we demonstrate this heuristic using ResNet architectures on CIFAR.

\end{abstract}

\maketitle
\acresetall

\section{Introduction}
\label{sec:introduction}

Owing to its simplicity and empirical successes, \ac{SGD} has become the de facto method for training a wide range of models in machine learning.
This paper examines the properties of \ac{SGD} in non-convex problems with the aim of answering the following questions: 
\begin{enumerate}
[(Q1)]
\item
Does \ac{SGD} \emph{always} converge?
\item
Does \ac{SGD} \emph{always} avoid spurious critical regions, such as non-isolated saddle points, etc.?
\item
How fast does \ac{SGD} converge to local minima as a function of the method's step-size policy?
\end{enumerate}
\smallskip

We provide the following precise answers to these questions:

\para{On (Q1):}
Under mild conditions for the function to be optimized, and allowing for a wide range of step-size schedules of the form $\Theta(1/\run^{\pexp})$ for $\pexp\in(0,1]$, the iterate sequence $\curr$ of \ac{SGD} converges with probability $1$.
In contrast to existing mean squared error guarantees of the form $\exof{\norm{\nabla\obj(\curr)}^{2}} \to 0$ (where $\obj$ is the problem's objective), this is a stronger, trajectory convergence result:
It is not a guarantee that holds on average, but a convergence certificate that applies with probability $1$ to \emph{any} instantiation of the algorithm.

\para{On (Q2):}
With probablity $1$, the trajectories of \ac{SGD} avoid all strict saddle manifolds \textendash\ \ie sets of critical points $\saddle$ with at least one negative Hessian eigenvalue ($\lambda_{\min}(\nabla^{2}\obj(\saddle)) < 0$).
Such manifolds include ridge hypersurfaces and other connected sets of non-isolated saddle points that are common in the loss landscapes of overparametrized neural networks \citep{LXTS+18}.
In this way, our result complements and extends a series of saddle avoidance results for \emph{deterministic} gradient descent \citep{LSJR16,LPPS+19,FVP19b,DJLJ+17,PPW19}, and with \emph{high probability} \citep{GHJY15} or \emph{in expectation} \citep{VS19} for \acl{SGD}.

\para{On (Q3):}
If \ac{SGD} is run with a step-size schedule of the form $\curr[\step] = \Theta(1/\run^{\pexp})$ for some $\pexp\in(0,1]$, it converges at a rate of $\exof{\norm{\curr - \sol}^{2}} = \bigoh(1/\run^{\pexp})$ to local minimizers that are regular in the sense of Hurwicz (\ie $\nabla^{2}\obj(\sol) \mg 0$).
We stress here that this is a ``last iterate'' convergence guarantee;
neither ergodic, nor of a mean-squared gradient norm type.
This is crucial for real-world applications because, in practice, \ac{SGD} training is based on the last generated point.

Taken together, the above suggests that a vanishing step-size policy has significant theoretical benefits:
almost sure convergence,
avoidance of spurious critical points (again with probability $1$),
and
fast stabilization to local minimizers.
We explore these properties in a range of standard non-convex test functions and by training a ResNet architecture for a classification task over CIFAR.

The linchpin of our approach is the \emph{\acs{ODE} method} of \acl{SA} as pioneered by \citet{Lju77,BMP90,KY97}, and \citet{Ben99}.
As such, our analysis combines a wide range of techniques from the theory of dynamical systems along with a series of martingale limit theory tools originally developed by \citet{Pem90} and \citet{BD96}.

\para{Related work}

Ever since the seminal paper of \citet{RM51}, \ac{SGD} has given rise to a vast corpus of literature that we cannot hope to do justice here.
We discuss below only those works which \textendash\ to the best of our knowledge \textendash\ are the most relevant to the contributions outlined above.

The first result on the convergence of \ac{SGD} trajectories is due to \citet{Lju77,Lju86}, who proved the method's convergence under the boundedness assumption $\sup_{\run} \norm{\curr} < \infty$.
Albeit intuitive, this assumption is fairly difficult to establish from first principles and the problem's primitives.
Because of this, boundedness has persisted in the \acl{SA} literature as a condition that needs to be enforced ``by hand'', see \eg \citet{Ben99,Bor08,KY97}, and references therein.
To rid ourselves of this condition, we resort to a series of shadowing arguments that interpolate between continuous and discrete time.
Our results also improve on a more recent result by \citet{BT00} who use a completely different analysis to dispense of boundedness via the use of more restrictive, rapidly decaying step-size policies.

On the issue of saddle-point avoidance, \citet{Pem90} and \citet{BD96} showed that \ac{SGD} avoids \emph{hyperbolic} saddle points ($\lambda_{\min}(\nabla^{2}\obj(\saddle)) < 0$, $\det \nabla^{2}\obj(\saddle) \neq 0$) with probability $1$.
More recently, and under different assumptions, \citet{GHJY15} showed that \ac{SGD} avoids \emph{strict} saddle points ($\lambda_{\min}(\nabla^{2}\obj(\saddle)) < 0$) with high probability,
whereas the work of \citet{VS19} guarantees escape from strict saddles in expectation.
By comparison, our paper shows that strict saddles are avoided \emph{with probability $1$}, thus providing the missing link between these two threads;
for completenes, we review these results in detail in \cref{sec:avoidance}.

The papers mentioned above should be disjoined from an extensive literature on saddle-point avoidance results for \emph{deterministic} gradient descent \citep{LSJR16,JGNK+17,LPPS+19,DJLJ+17,FVP19b,PPW19,PP17}.
Given that these works focus exclusively on deterministic methods, they have no bearing on our work here.

Finally, regarding the rate of convergence of \ac{SGD} in non-convex problems, \citet{GL13,GL16} established a series of bounds of the form $\exof{ \norm{\nabla\obj(X_R)}^{2}} = \bigoh(1/\sqrt{\nRuns})$, where $R$ is drawn randomly from the running horizon $\{1,\dotsc,\nRuns\}$ of the process.
More recently, \citet{LHLT19} provided a non-asymptotic rate analysis for $\alpha$-Holder smooth functions, without a bounded gradient assumption;
specifically, \citet{LHLT19} proved that, for some $T$, $\min_{\run\leq\nRuns} \exof{ \norm{ \nabla \obj (\curr) }^2 } = \bigoh(\nRuns^{\pexp-1})$ with stepsize $\curr[\step] = \step / \run^{\pexp}$ and $\pexp \in (1 / (1 + \alpha), 1)$.
There is no overlap of our results or analysis with these works, and we are not aware of convergence guarantees similar to our own in the literature.

\para{Notation}

In the rest of our paper, $\vecspace$ denotes a $\vdim$-dimensional Euclidean space.
We also write
$\braket{\cdot\,}{\cdot}$ for the inner product on $\vecspace$,
$\norm{\cdot}$ for the induced norm,
and
$\sphere^{\vdim-1} = \setdef{\point \in \points}{\norm{\point} = 1}$ for the unit hypersphere of $\points$.
Since the space is Euclidean, we make no distinction between primal and dual vectors (or norms).

\section{Problem setup and assumptions}
\label{sec:setup}

\subsection{Problem setup}
\label{sec:assumptions}

Throughout the sequel, we focus on the non-convex optimization problem
\begin{equation}
\label{eq:opt}
\tag{Opt}
\begin{aligned}
\textrm{minimize}_{x\in\points}
	& \
	\obj(\point),
\end{aligned}
\end{equation}
where $\obj\from\points\to\R$ is a $\vdim$-times differentiable function satisfying the following blanket assumptions.

\begin{assumption}
\label{asm:reg}
$\obj$ is \emph{$\lips$-Lipschitz and $\smooth$-smooth}, \ie
\begin{equation}
\dnorm{\nabla\obj(\point)}
	\leq \lips
	\quad
	\textrm{and}
	\;\;
\dnorm{\nabla\obj(\pointalt) - \nabla\obj(\point)}
	\leq \smooth \norm{\pointalt - \point}
	\;\;
	\text{for all $\point,\pointalt\in\points$}.
\end{equation}
\end{assumption}

%

\begin{assumption}
\label{asm:sublvl}
The \emph{sublevels}
\(
\flvl_{c}
	\equiv \setdef{\point\in\points}{\obj(\point) \leq c}
\)
of $\obj$ are bounded for all $c < \sup\obj$.
\end{assumption}

\begin{assumption}
\label{asm:gradlvl}
The \emph{gradient sublevels}
\(
\glvl_{\eps}
	\equiv \setdef{\point\in\points}{\dnorm{\nabla\obj(\point)} \leq \eps}
\)
of $\obj$
are bounded for some $\eps>0$.
\end{assumption}

\Cref{asm:reg,asm:sublvl,asm:gradlvl} are fairly standard in non-convex analysis and optimization.
Taken individually,
\cref{asm:reg} is a basic
regularity requirement for $\obj$;
\cref{asm:sublvl} guarantees the existence of solutions to \eqref{eq:opt} by ruling out vacuous cases like $\obj(\point) = -\point$;
and, finally,
\cref{asm:gradlvl} serves to exclude objectives with near-critical behavior at infinity such as  $\obj(\point) = -e^{-\point^{2}}$.%
\footnote{Note that \cref{asm:gradlvl} only concerns near-critical points, not regions where $\dnorm{\nabla\obj(\point)}$ may be large.}
Taken together, \cref{asm:reg,asm:sublvl,asm:gradlvl} further imply that the critical set
\begin{equation}
\label{eq:crit}
\sols
	\equiv
	\crit(\obj)
	= \setdef{\point\in\points}{\nabla\obj(\point) = 0}
\end{equation}
of $\obj$ is nonempty, a fact that we use freely in the sequel.

Typical examples of \eqref{eq:opt} in machine learning comprise
neural networks with sigmoid activation functions,
underdetermined inverse problems,
empirical risk minimization models,
etc.
In such problems, obtaining accurate gradient input is impractical, so to solve \eqref{eq:opt}, we often rely on \acli{SG} information, obtained for example by taking a mini-batch of training instances.

\subsection{Assumptions on the oracle}
\label{sec:oracle}

With this in mind, we will assume throughout that the optimizer can access $\nabla\obj$ via a \acdef{SFO}.
Formally, this is a black-box feedback mechanism which, when queried at an input point $\point\in\points$, returns a random vector $\orcl(\point;\sample)$ with $\sample$ drawn from some (complete) probability space $\probspace$.
In more detail, decomposing the oracle's output at $\point$ as
\begin{equation}
\label{eq:SFO}
\tag{SFO}
\orcl(\point;\sample)
	= \nabla\obj(\point)
	+ \noise(\point;\sample)
	\,,
\end{equation}
we  make the following assumption.
\begin{assumption}
\label{asm:noise}
The \emph{error term} $\noise(\point;\sample)$ of \eqref{eq:SFO} has
\begin{subequations}
\label{eq:noise}
\begin{alignat}{2}
\label{eq:zeromean}
\textpar{\textit a}
\quad
	&\emph{Zero mean:}
	&\quad
	&\exof{\noise(\point;\sample)}
		= 0
	\\
\textpar{\textit b}
\quad
\label{eq:moment}
	&\emph{Finite $\qexp$-th moments:}
	&\quad
	&\exof{\dnorm{\noise(\point;\sample)}^{\qexp}}
		\leq \noisepar^{\qexp}
	\;\
	\text{for some $\qexp\geq2$ and $\noisepar\geq 0$}.
	\hspace{2.5em}
\end{alignat}
\end{subequations}
\end{assumption}

\Cref{asm:noise} is standard in stochastic optimization and is usually stated with $\qexp = 2$, \ie as a ``finite variance'' condition, \cf \citet{Nes04,Pol87,JNT11,Ben99}, and many others.
Allowing values of $\qexp$ greater than $2$ provides more flexibility in the choice of step-size policies, so we  keep \eqref{eq:moment} as a blanket assumption throughout.
We  also formally allow the value $\qexp=\infty$ in \eqref{eq:moment}, in which case we will say that the noise is \emph{bounded in $L^{\infty}$};
put simply, this corresponds to the standard assumption that the noise in \eqref{eq:SFO} is bounded almost surely.


\subsection{\Acl{SGD}}
\label{sec:SGD}

With all this in hand, the \acf{SGD} algorithm can be written as
\begin{equation}
\label{eq:SGD}
\tag{SGD}
\next
	= \curr
	- \curr[\step] \curr[\signal].
\end{equation}
In the above,
$\run = \running$ is the algorithm's iteration counter,
$\curr[\step]$ is the algorithm's step-size,
and
$\curr[\signal]$ is a sequence of gradient signals of the form
\begin{equation}
\label{eq:signal}
\curr[\signal]
	= \orcl(\curr;\curr[\sample])
	= \nabla\obj(\curr)
	+ \curr[\noise].
\end{equation}
Each gradient signal $\curr[\signal]$ is generated by querying the oracle at $\curr$ with some random seed $\curr[\sample]$.
For concision, we write $\curr[\noise] \equiv \noise(\curr,\curr[\sample])$ for the gradient error at the $\run$-th iteration
and
$\curr[\filter] = \sigma(\state_{\start},\dotsc,\curr)$ for the natural filtration of $\curr$;
in this notation, $\curr[\sample]$ and $\curr[\signal]$ are \emph{not} $\curr[\filter]$-measurable.

All our results for \eqref{eq:SGD} are stated in the framework of the basic assumptions above.
The price to pay for this degree of generality is that the analysis requires an intricate interplay between martingale limit theory and the theory of \acl{SA};
we review the relevant notions below.

\section{Stochastic approximation}
\label{sec:SA}

\para{Asymptotic pseudotrajectories}

The departure point for our analysis is to rewrite the iterates of \eqref{eq:SGD} as
\(
\left(\next - \curr\right) / \curr[\step]
	= \nabla(\obj(\curr)) + \curr[\noise].
\)
In this way, \eqref{eq:SGD} can be seen as a \acl{RM} discretization of the continuous-time \acli{GD}
\begin{equation}
\label{eq:GD}
\tag{GD}
\dorbit{\ctime}
	= - \nabla\obj(\orbit{\ctime}).
\end{equation}
The main motivation for this comparison is that $\obj$ is a strict Lyapunov function for \eqref{eq:GD}, indicating that its solution orbits converge to the critical set $\sols$ of $\obj$ (see the supplement for a formal statement and proof of this fact).
As such, if the trajectories of \eqref{eq:SGD} are ``good enough'' approximations of the solutions of \eqref{eq:GD}, one would expect \eqref{eq:SGD} to enjoy similar convergence properties.

To make this idea precise, we first connect continuous and discrete time by letting $\curr[\proper] = \sum_{\runalt=\start}^{\run} \iter[\step]$ denote the time that has ``elapsed'' for \eqref{eq:SGD} up to iteration counter $\run$ (inclusive); that is,  a step-size in discrete time is translated to elapsed time in the continuous case, and vice-versa.
We may then define the continuous-time interpolation of an iterate sequence $\curr$ of \eqref{eq:SGD} as
\begin{equation}
\label{eq:interpolation}
\apt{\ctime}
	= \curr
		+ \left[ (\ctime - \curr[\proper]) / (\next[\proper] - \curr[\proper]) \right] (\next - \curr)
		\qquad
		\text{for all $\ctime \in [\curr[\proper],\next[\proper]]$}.
\end{equation}
To compare this trajectory to the solutions of \eqref{eq:GD}, we further need to define the ``flow'' of \eqref{eq:GD} which describes how an ensemble of initial conditions evolves over time.
Formally, we let $\flowmap\from\R_{+}\times\points\to\points$ denote the map which sends an initial $\point\in\points$ to the point $\flow{\ctime}{\point}\in\points$ by following for time $\ctime\in\R_{+}$ the solution of \eqref{eq:GD} starting at $\point$.
We then have the following notion of ``asymptotic closeness'' between a sequence generated by \eqref{eq:SGD} and the flow of the dynamics \eqref{eq:GD}:

\begin{definition}
\label{def:APT}
We say that $\apt{\ctime}$ is an \acdef{APT} of \eqref{eq:GD} if, for all $\horizon > 0$:
\begin{equation}
\label{eq:APT}
\txs
\lim_{\ctime\to\infty}
	\sup_{0 \leq h \leq \horizon}
		\norm{\apt{\ctime + h} - \flow{h}{\apt{\ctime}}}
	= 0,
\end{equation}
\end{definition}

The notion of an \ac{APT} is due to \citet{BH96} and essentially posits that $\apt{\ctime}$ tracks the flow of \eqref{eq:GD} with arbitrary accuracy over windows of arbitrary length as $\ctime\to\infty$.
When this is the case, we will slightly abuse terminology and say that the sequence $\curr$ itself comprises an \ac{APT} of \eqref{eq:GD}.

With all this in hand, the formal link between \eqref{eq:GD} and \eqref{eq:SGD} is as follows:

\begin{restatable}{proposition}{propAPT}
\label{prop:APT}
Suppose that \cref{asm:reg,asm:noise} hold
and
\eqref{eq:SGD} is employed with a step-size sequence such that $\sum_{\run=\start}^{\infty} \curr[\step] = \infty$ and $\sum_{\run=\start}^{\infty} \curr[\step]^{1+\qexp/2} < \infty$ with $\qexp\geq2$ as in \cref{asm:noise}.
Then, with probability $1$, $\curr$ is an \ac{APT} of \eqref{eq:GD}.
\end{restatable}

\begin{corollary}
\label{cor:APT-step}
Suppose that \eqref{eq:SGD} is run with $\curr[\step] = \Theta(1/\run^{\pexp})$ for some $\pexp \in (2/(\qexp + 2),1]$ and assumptions as in \cref{prop:APT}.
Then, with probability $1$, $\curr$ is an \ac{APT} of \eqref{eq:GD}.
\end{corollary}

This comparison result plays a key role in the sequel because it delineates the range of step-size policies under which the discrete-time system \eqref{eq:SGD} is well-approximated by the continuous-time dynamics \eqref{eq:GD}.
We discuss this issue in detail in the next section.

\section{Convergence analysis}
\label{sec:analysis}

Heuristically, the goal of approximating \eqref{eq:SGD} via \eqref{eq:GD} is to reduce the difficulty of the direct analysis of the former by leveraging the strong convergence properties of the latter.
of the latter, which is relatively straightforward to analyze, to the former, which is much more difficult.
However, the notion of an \ac{APT} does not suffice in this regard:
in the supplement, we provide an example where a discrete-time \ac{APT} has a completely different behavior relative to the underlying flow.
As such, a considerable part of our analysis below focuses on tightening the guarantees provided by the \ac{APT} approximation scheme.

\subsection{Boundedness and stability of the approximation}
\label{sec:stability}

The basic point of failure in the \acl{SA} approach is that \acp{APT} may escape to infinity, rendering the whole scheme useless, \cf \citep{Ben99,Bor08} and references therein.
It is for this reason that a large part of the literature on \ac{SGD} explicitly assumes that the trajectories of the process are bounded (precompact), \ie
\begin{equation}
\label{eq:stability}
\txs
\sup_{\ctime\geq\cstart} \norm{\apt{\ctime}}
	< \infty
	\quad
	\as.
\end{equation}
However, this is a prohibitively strong assumption for \eqref{eq:opt}:
unless certified ahead of time, any theoretical result relying on this assumption would be of limited practical value.

Our first result below provides exactly this certification by establishing that \eqref{eq:stability} is solely an implication of our underlying \crefrange{asm:reg}{asm:noise}.
It is a non-trivial outcome which provides the key to unlocking the potential of \acl{SA} techniques in the sequel.

\begin{restatable}{theorem}{stability}
\label{thm:stability}
Suppose that \crefrange{asm:reg}{asm:noise} hold
and
\eqref{eq:SGD} is run with a variable step-size sequence
of the form $\curr[\step]\propto 1/\run^{\pexp}$ for some $\pexp \in (2/(\qexp + 2),1]$.
Then, with probability $1$, every \ac{APT} $\apt{\ctime}$ of \eqref{eq:GD} that is induced by \eqref{eq:SGD} has $\sup_{\ctime\geq\cstart} \norm{\apt{\ctime}} < \infty$.
\end{restatable}

Because of the generality of our assumptions, the proof of \cref{thm:stability} involves a delicate combination of  non-standard techniques;
for completeness, we provide a short sketch below and refer the reader to the supplement for the details.

\begin{proof}[Sketch of proof of \cref{thm:stability}]
The main reasoning evolves along the following lines:

\begin{enumerate}[Step 1.,leftmargin=*]
\item
We first show that, under the stated assumptions, there exists a \textpar{possibly random} subsequence $\state_{\run_{\runalt}}$ of $\curr$ that converges to $\sols$;
formally, $\liminf_{\run\to\infty} \dist(\curr,\sols) = 0$ \as.
As a result, $\apt{\ctime}$ eventually reaches a sublevel set $\flvl_{\eps}$ whose elements are arbitrarily close to $\sols$, \ie there exists some $\ctime_{\eps} > 0$ such that $\apt{\ctime_{\eps}} \in \flvl_{\eps}$.

\item
By a technical argument relying on the regularity assumptions for $\obj$ (\cf \cref{asm:reg}), it can be shown that there exists some \emph{uniform} time window $\tau$ such that $\apt{\ctime}$ remains within uniformly bounded distance to $\flvl_{\eps}$ for all $\ctime \in [\ctime_{\eps},\ctime_{\eps} + \tau]$.
Thus, once $\apt{\ctime}$ gets close to $\flvl_{\eps}$, it will not escape too far within a fixed length of time.

\item
An additional technical argument reveals that, under the stated assumptions for $\obj$, the trajectories of \eqref{eq:GD} either descend the objective by a uniform amount, or they have reached a neighborhood of the critical set $\sols$ where further descent is impossible (or irrelevant).

\item
By combining the two previous steps, we conclude that $\apt{\ctime_{\eps} + \tau} \in \flvl_{\eps}$ at the end of said window.
This argument may then be iterated ad infinitum to show inductively that $\apt{\ctime} \in \flvl_{\eps}$ for all intervals of the form $[\ctime_{\eps} + \runalt\tau, \ctime_{\eps} + (\runalt+1)\tau]$.
\end{enumerate}

Since $\flvl_{\eps}$ is bounded (by \cref{asm:sublvl}), we conclude that $\apt{\ctime}$ remains in a compact set for all $\ctime\geq0$, \ie $\apt{\ctime}$ is precompact.
The conclusion of \Cref{thm:stability} then follows by \cref{cor:APT-step}.
\end{proof}

\subsection{Almost sure convergence}
\label{sec:convergence}

By virtue of \cref{thm:stability}, we are now in a position to state our almost sure convergence result:

\begin{theorem}
\label{thm:convergence}
Suppose that \crefrange{asm:reg}{asm:noise} hold
and
\eqref{eq:SGD} is run with a variable step-size sequence
of the form $\curr[\step] = \Theta(1/\run^{\pexp})$ for some $\pexp \in (2/(\qexp + 2),1]$.
Then, with probability $1$, $\curr$ converges to a \textpar{possibly random} connected component $\sols_{\infty}$ of $\sols$ over which $\obj$ is constant.
\end{theorem}

\begin{corollary}
\label{cor:convergence}
With assumptions as in \cref{thm:convergence}, we have the following:
\begin{enumerate}
\item
$\obj(\curr)$ converges \as to some critical value $\obj_{\infty}$.
\item
Any limit point of $\curr$ is \as a critical point of $\obj$.
\end{enumerate}
\end{corollary}

\Cref{thm:convergence} extends a range of existing treatments of \eqref{eq:SGD} under explicit boundedness assumptions of the form \eqref{eq:stability}, \cf \citep{Lju77,Ben99,Bor08} and references therein.
It also improves on a similar result by \citet{BT00} who use a completely different analysis to dispense of boundedness requirements via the use of more restrictive step-size policies.
Specifically, \citet{BT00} require the \acl{RM} summability conditions $\sum_{\run} \curr[\step] = \infty$ and $\sum_{\run} \curr[\step]^{2} < \infty$ under a bounded variance assumption.
In this regard, our analysis extends to more general step-size policies, while that of \citet{BT00} cannot because of its reliance on the Robbins-Siegmund theorem for almost-supermartingales \citep{RS71}.
Among other benefits, this added degree of flexibility is a key advantage of the \ac{APT} approach.


The heavy lifting in the proof of \cref{thm:convergence} is provided by \cref{thm:stability}.
Thanks to this boundedness certificate,
the total chain of implications is relatively short, so we provide it in full below.

\begin{proof}[Proof of \cref{thm:convergence}]
Under the stated assumptions, $\obj$ is a strict Lyapunov function for \eqref{eq:GD} in the sense of \citet[Chap.~6.2]{Ben99}. Specifically, this means that $\obj(\flow{\ctime}{\point})$ is strictly decreasing in $\ctime$ unless $\point$ is a stationary point of \eqref{eq:GD}.
Furthermore, by Sard's theorem \citep[Chap.~2]{Mil65}, the set $\obj(\sols)$ of critical values of $\obj$ has Lebesgue measure zero \textendash\ and hence, empty topological interior.
Therefore, applying Theorem 5.7 and Proposition 6.4 of \citet{Ben99} in tandem, we conclude that any precompact \acl{APT} of \eqref{eq:GD} converges to a connected component $\sols_{\infty}$ of $\sols$ over which $\obj$ is constant.
Since \cref{thm:stability} guarantees that the \acp{APT} of \eqref{eq:GD} induced by \eqref{eq:SGD} are bounded with probability $1$, our claim follows.
\end{proof}

\subsection{Avoidance analysis}
\label{sec:avoidance}

\Cref{thm:convergence} represents a strong convergence guarantee but, at the same time, it does not characterize the component of $\sols$ to which $\curr$ converges.
The rest of this section is devoted to showing that $\curr$ does not converge to a component of $\sols$ that only consists of saddle points (a \emph{saddle-point manifold}).
Specifically, we will make precise the following informal statement:
\begin{center}
\emph{\eqref{eq:SGD} avoids strict saddles \textendash\ and sets thereof \textendash\ with probability $1$}.
\end{center}
To set the stage for the analysis to come, we begin by reviewing some classical and recent results on the avoidance of saddle points.
We then present our general results towards the end of the section.

To begin, a crucial role will be played in the sequel by the \emph{Hessian matrix} of $\obj$, viz.
\begin{equation}
\label{eq:Hessian}
\hmat(\point)
	\equiv \nabla^{2}\obj(\point)
	\equiv (\pd_{i} \pd_{j} \obj(\point))_{i,j=1,\dotsc,\vdim}.
\end{equation}
Since $\hmat(\point)$ is symmetric, all of its eigenvalues are real.
If $\saddle$ is a critical point of $\obj$ and
$\lambda_{\min}(\hmat(\saddle)) < 0$, we say that $\saddle$ is a \emph{strict saddle point} \citep{LSJR16,LPPS+19}.

By standard results in center manifold theory \citep{Tes12}, the space around strict saddle points admits a decomposition into a \emph{stable}, \emph{center} and \emph{unstable} manifold (each of the former two possibly of dimension zero; the latter of dimension at least $1$ given that $\lambda_{\min}(\hmat(\saddle)) < 0$).
Heuristically, under the continuous-time dynamics \eqref{eq:GD}, directions along the stable manifold of $\saddle$ are attracted to $\saddle$ at a linear rate, while those along the unstable manifold are repelled (again at a linear rate);
the dynamics along the center manifold could be considerably more complicated, but, in the presence of unstable directions, they only emerge from a measure zero of initial conditions.
As a result, if $\saddle$ is a strict saddle point of $\obj$, it stands to reason that \eqref{eq:SGD} should ``probably'' avoid it as well.

In the case of \emph{deterministic} gradient descent with step-size $\step < 1/\smooth$, this intuition was made precise by \citet{LSJR16,LPPS+19} who proved that all but a measure zero of initializations of gradient descent avoid strict saddles.
As we discussed in the introduction, this result was then extended to various deterministic settings, with different assumptions for the gradient oracle, the method's step-size, or the structure of the saddle-point manifold, see \eg \cite{JGNK+17,LPPS+19,DJLJ+17,FVP19b,PP17,PPW19} and the references therein.

In the stochastic regime, the situation is considerably more involved.
\citet{Pem90} and \citet{BD96} were the first to establish the avoidance of hyperbolic unstable equilibria in general \acl{SA} schemes.
However, a key requirement in the analysis of these works is that of \emph{hyperbolicity}, which in our setting amounts to asking that $\hmat(\saddle)$ is \emph{invertible}.
In particular, this means the saddle point in question cannot be isolated, nor can it have a center manifold:
both hypotheses are too stringent for applications of \ac{SGD} to contemporary machine learning models, such as deep net training, so their results do not apply in many cases of practical interest.

More relevant for our purposes is the recent result of \citet{GHJY15}, who provided the following guarantee.
Suppose that $\obj$ is \emph{$(\alpha,\beta,\eps,\delta)$-strict saddle}, \ie for all $\point\in\points$, one of the following holds:
\begin{enumerate*}
[(\itshape i\hspace*{.5pt}\upshape)]
\item
$\norm{\nabla\obj(\point)} \geq \eps$;
\item
$\lambda_{\min}(\hmat(\point)) \leq - \beta$;
or
\item
$\point$ is $\delta$-close to a local minimum $\point_{c}$ around which $\obj$ is $\alpha$-strongly convex.
\end{enumerate*}
Suppose further that $\obj$ is bounded, $\smooth$-Lipschitz smooth, and $\hmat(\point)$ is $\rho$-Lipschitz continuous;
finally, assume that the noise in the gradient oracle \eqref{eq:SFO} is finite \as and contains a component uniformly sampled from the unit sphere.
Then, given a confidence level $\zeta>0$, and assuming that \eqref{eq:SGD} is run with \emph{constant} step-size $\step = \bigoh(1/\log(1/\zeta))$, the algorithm produces after a given number of iterations a point which is $\bigoh(\sqrt{\gamma \log(1/(\gamma\zeta))})$-close to $\point_{c}$, and hence away from any strict saddle of $\obj$, with probability at least $1-\zeta$.

In a more recent paper, \citet{VS19} examined the convergence of \eqref{eq:SGD} to second-order stationary points.
More precisely, they showed that \eqref{eq:SGD} guarantees \emph{expected} descent for strict saddle points in a finite number of iterations, and with high probability, \eqref{eq:SGD} iterates reach a set of \emph{approximate} second-order stationary points in finite time.

The theory of \citet{Pem90} and the result of \citet{GHJY15} paint a complementary picture to the above:
\citet{Pem90} shows that saddle points are avoided with probability $1$, provided they are hyperbolic (\ie $\det \nabla^{2}\obj(\saddle) \neq 0$);
on the other hand, \citet{GHJY15} require much less structure on the saddle point, but they only provide a result with high probability (and $\zeta$ cannot be taken to zero because the range of allowable step-sizes would also vanish).%
\footnote{\citet{Pem90} employs a vanishing step-size, which is more relevant for us:
\eqref{eq:SGD} with persistent noise and a constant step-size is an irreducible ergodic Markov chain whose trajectories do not converge \emph{anywhere} \citep{Ben99}.}
Our objective in the sequel is to provide a result that combines the ``best of both worlds'', \ie almost sure avoidance of strict saddle points (and sets thereof) with probability $1$.

To that end, we make the following assumption for the noise:
\begin{assumption}
\label{asm:exciting}
The error term $\noise \equiv \noise(\point;\sample)$ of \eqref{eq:SFO} is \emph{uniformly exciting}, \ie there exists some $\const>0$ such that
\begin{equation}
\label{eq:exciting}
\exof{\braket{\noise(\point;\sample)}{\unitvec}^{+}}
	\geq \const
\end{equation}
for all $\point\in\points$ and all unit vectors $\unitvec \in \sphere^{\vdim-1}$.
\end{assumption}

This assumption simply means that the average projection of the noise along every ray in $\points$ is uniformly positive;
in other words, $\noise$ ``excites'' all directions uniformly \textendash\ though not necessarily \emph{isotropically}.
As such, \cref{asm:exciting} is automatically satisfied by noisy \acl{GD} (\eg as in \citet{GHJY15}), generic finite sum objectives with at least $\vdim$ summands, etc.


\begin{figure}[tbp]
\centering
\includegraphics[width=.5\textwidth]{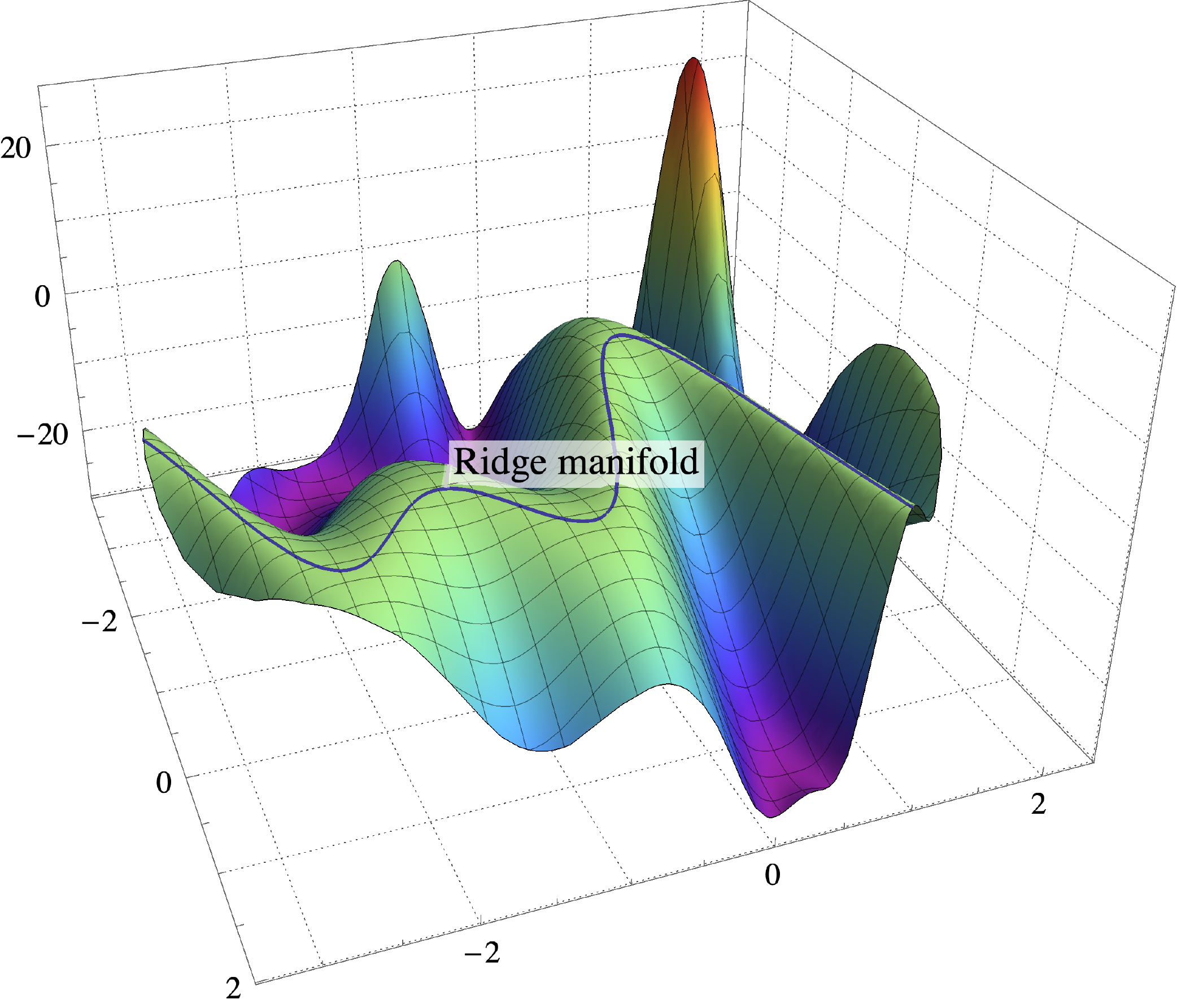}
\caption{A strict saddle manifold (a ridge), typical of ResNet loss landscapes \citep{LXTS+18}.}
\label{fig:saddle}
\end{figure}


With all this in hand, we say that $\saddles$ is a \emph{strict saddle manifold} of $\obj$ if it is a smooth connected component of $\sols$ such that:
\begin{enumerate}
\item
Every $\saddle\in\saddles$ is a strict saddle point of $\obj$ \textpar{\ie $\lambda_{\min}(\hmat(\saddle)) < 0$}.
\item
There exist $\const_{-},\const_{+}>0$ such that, for all $\saddle\in\saddles$, all negative eigenvalues of $\hmat(\saddle)$ are bounded from above by $-\const_{-} < 0$, and any positive eigenvalues (\emph{if} they exist) are bounded from below by $\const_{+}$.
\end{enumerate}

Somewhat informally, the definition of a strict saddle manifold implies that the eigenspaces of $\hmat(\saddle)$ corresponding to zero, positive, and negative eigenvalues decompose smoothly along $\saddles$
and $\saddles$ can be seen as an ``integral manifold'' of the nullspace of the Hessian of $\obj$.

With all this in hand, we are finally in a position to state our main avoidance result.

\begin{restatable}{theorem}{avoidance}
\label{thm:avoidance}
Suppose that \eqref{eq:SGD} is run with a variable step-size sequence of the form $\curr[\step]\propto 1/\run^{\pexp}$ for some $\pexp \in (0,1]$.
If
\crefrange{asm:reg}{asm:exciting} hold \textpar{with $\qexp=\infty$ for \cref{asm:noise}},
and
$\saddles$ is a strict saddle manifold of $\obj$,
we have
\(
\probof{\curr \to \saddles \; \text{as $\run\to\infty$}}
	= 0.
\)
\end{restatable}
\Cref{thm:avoidance} is the formal version of the avoidance principle that we stated in the beginning of this section.
Importantly, it makes \emph{no} assumptions regarding the initialization of \eqref{eq:SGD} and holds for \emph{any} initial condition.

The proof of \cref{thm:avoidance} relies on two basic components.
The first is a probabilistic estimate, originally due to \citet{Pem90}, that shows that a certain class of stochastic processes avoid zero with probability $1$.
The second is a differential-geometric argument, building on \citet{BH95} and \citet{Ben99}, and relying on center manifold theory to isolate the center/stable and unstable manifolds of $\saddles$.
Combining these two components, it is possible to show that even ambulatory random walks along the stable manifold of $\saddles$ will eventually be expelled from a neighborhood of $\saddles$.
We provide the details of this argument in the paper's supplement.

\subsection{Rate of convergence}
\label{sec:rate}

We conclude our analysis of \eqref{eq:SGD} by establishing the algorithm's rate of convergence, as stated in \cref{thm:rate} below.
Since $\obj$ is non-convex, any convergence rate analysis of this type must be a fortiori local;
in view of this, we will examine the algorithm's convergence to local minimizers $\sol\in\sols$ that are regular in the sense of Hurwicz, \ie $\hmat(\sol) \mg 0$.

Because we are primarily interested in the convergence of the algorithm's trajectories, we focus here on the distance $\curr[\breg] = \norm{\curr - \sol}^{2}/2$ between the iterates of \eqref{eq:SGD} and a local minimizer of $\obj$.
In this light, our rate guarantee (which we state below), differs substantially from other results in the literature, in both scope and type, as it does not concern the ergodic average $\bar\state_{\run} = \run^{-1} \sum_{\runalt=\start}^{\run} \iter$ or the ``best iterate'' $\curr^{\text{best}} = \argmin_{\runalt=\start,\dotsc,\run} \norm{\nabla\obj(\iter)}$ of \eqref{eq:SGD}:
the former has very weak convergence in convex settings (if at all), while the latter cannot be calculated with access to perfect gradient information for the entire run of the process (in which case, stochastic \acl{GD} would become \emph{ordinary} \acl{GD}).


\begin{restatable}{theorem}{rate}
\label{thm:rate}
Fix some tolerance level $\delta>0$,
let $\sol$ be a regular minimizer of $\obj$,
and
suppose that \cref{asm:noise} holds.
Assume further that \eqref{eq:SGD} is run with a step-size schedule of the form $\curr[\step] = \step/(\run + \offset)^{\pexp}$ for some $\pexp \in (2/(\qexp + 2),1]$ and large enough $\offset,\step > 0$.
Then:
\begin{enumerate}
\setlength{\itemsep}{\smallskipamount}
\item
There exist neighborhoods $\nhd$ and $\nhd_{1}$ of $\sol$ such that, if $\init\in\nhd_{1}$, the event
\begin{equation}
\label{eq:stay}
\event_{\nhd}
	= \{\curr\in\nhd\;\text{for all $\run=\running$}\}
\end{equation}
occurs with probability at least $1-\delta$.
\item
Conditioned on $\event_{\nhd}$, we have
\begin{equation}
\label{eq:rate}
\exof{\norm{\curr - \sol}^{2} \given \event_{\nhd}}
	 = \bigoh(1/\run^{\pexp}).
\end{equation}
\end{enumerate}
\end{restatable}

\begin{remark*}
Note that \cref{thm:rate} does not presuppose \cref{asm:reg,asm:sublvl,asm:gradlvl};
since the rate analysis is local, the differentiability of $\obj$ suffices.
\end{remark*}


The proof of \cref{thm:rate} relies on showing that
\begin{enumerate*}
[\itshape a\upshape)]
\item
$\sol$ is \emph{stochastically stable}, \ie with high probability, any initialization that is close enough to $\sol$ remains close enough;
and
\item
conditioned on this event, the distance $\curr[\breg] = (1/2) \norm{\curr - \sol}^{2}$ to a regular local minimizers behaves as an ``almost'' supermartingale.
\end{enumerate*}
A major complication that arises here is that this conditioning changes the statistics of the noise, so the martingale property ceases to hold.
Overcoming this difficulty requires an intricate probablistic argument that we present in the supplement (where we also provide explicit expressions of the constants in the estimate of \cref{thm:rate}).

\section{Numerical experiments}
\label{sec:numerics}

As an illustration of our theoretical analysis, we plot in \cref{fig:Shekel} the convergence rate of \eqref{eq:SGD} in the standard Shekel risk benchmark function $\obj(\point) = \sum_{i=1}^{N} \big[\sum_{j=1}^{\vdim} (\point_{j} - a_{ij})^{2} + c_{i}\big]^{-1}$ where $A = (a_{ij})$ is a skew data matrix and $c = (c_{1},\dotsc,c_{N})$ is a bias vector of dimension $\vdim=500$ \citep{JY13}.
For our experiments, we ran $N=10^{3}$ instances of \eqref{eq:SGD} with a constant, $1/\sqrt{\run}$, and $1/\run$ step-size schedule, and we plotted the value difference $\obj(\curr) - \obj_{\infty}$ of the sample average (marked black lines) and the min-max spread of the samples for a 95\% confidence level region (shaded green, red and blue respectiely for the constant, $1/\sqrt{\run}$ and $1/\run$ policies respectively).
The constant step-size schedule initially performs better, but quickly saturates and is overcome by the $1/\run$ schedule;
overall, the $1/\run$ policy converges faster than the other two by $2$ to $4$ orders of magnitude.

Coupled with our theoretical results, these tests suggest that a vanishing step-size policy could have significant advantages when used for training machine learning models.
The key drawback to this approach is that a rapidly vanishing step-size could cause the algorithm to traverse the loss landscape at a very slow pace and/or get trapped at inferior local minima.
However, it also provides a sound theoretical justification for the following ``best of both worlds'' training heuristic:
given a budget of gradient iterations, run \ac{SGD} with a constant step-size for a fraction of this budget, and then implement a ``cooldown'' phase with a vanishing step-size for the rest.
We demonstrate the benefits of this ``cooldown'' heuristic in a standard ResNet18 architecture for a classification task over CIFAR10.
In particular, in \cref{fig:ResNet}, we ran \eqref{eq:SGD} with a constant step-size for $100$ epochs, with checkpoints at different cutoffs;
then, at each checkpoint, we launched the ``cooldown'' period with step-size $1/\run$. \cref{fig:ResNet} demonstrates the improvement due to the cool-off period over the training loss:
specifically, it shows that it is always beneficial to run the last training epochs with a vanishing step-size.


\begin{figure*}[t]
\centering
\footnotesize
\begin{subfigure}{.46\linewidth}
\centering
\includegraphics[height=5cm]{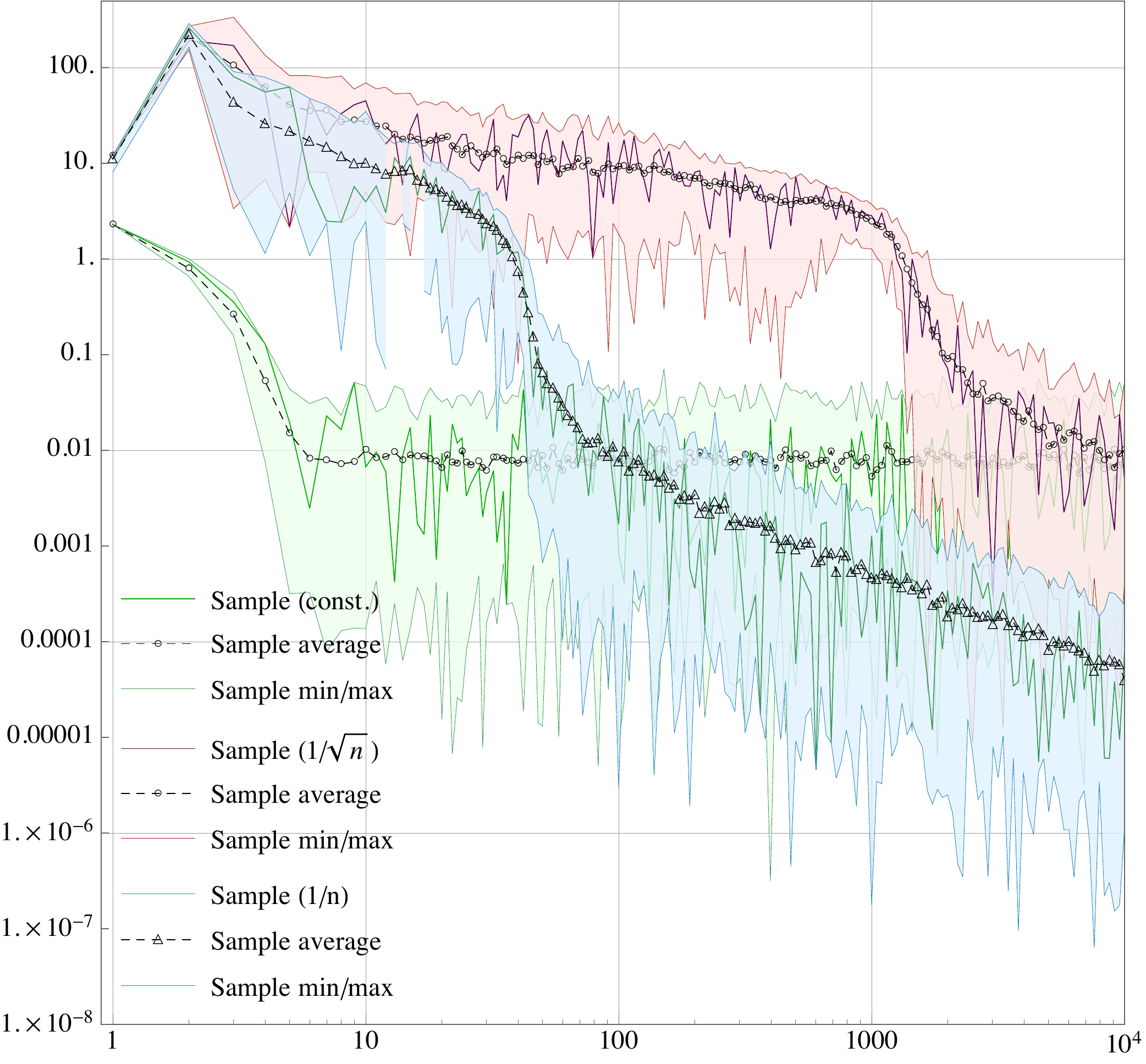}%
\caption{Speed of convergence in the Shekel benchmark.}
\label{fig:Shekel}
\end{subfigure}
\qquad
\begin{subfigure}{.46\linewidth}
\centering
\includegraphics[height=5cm]{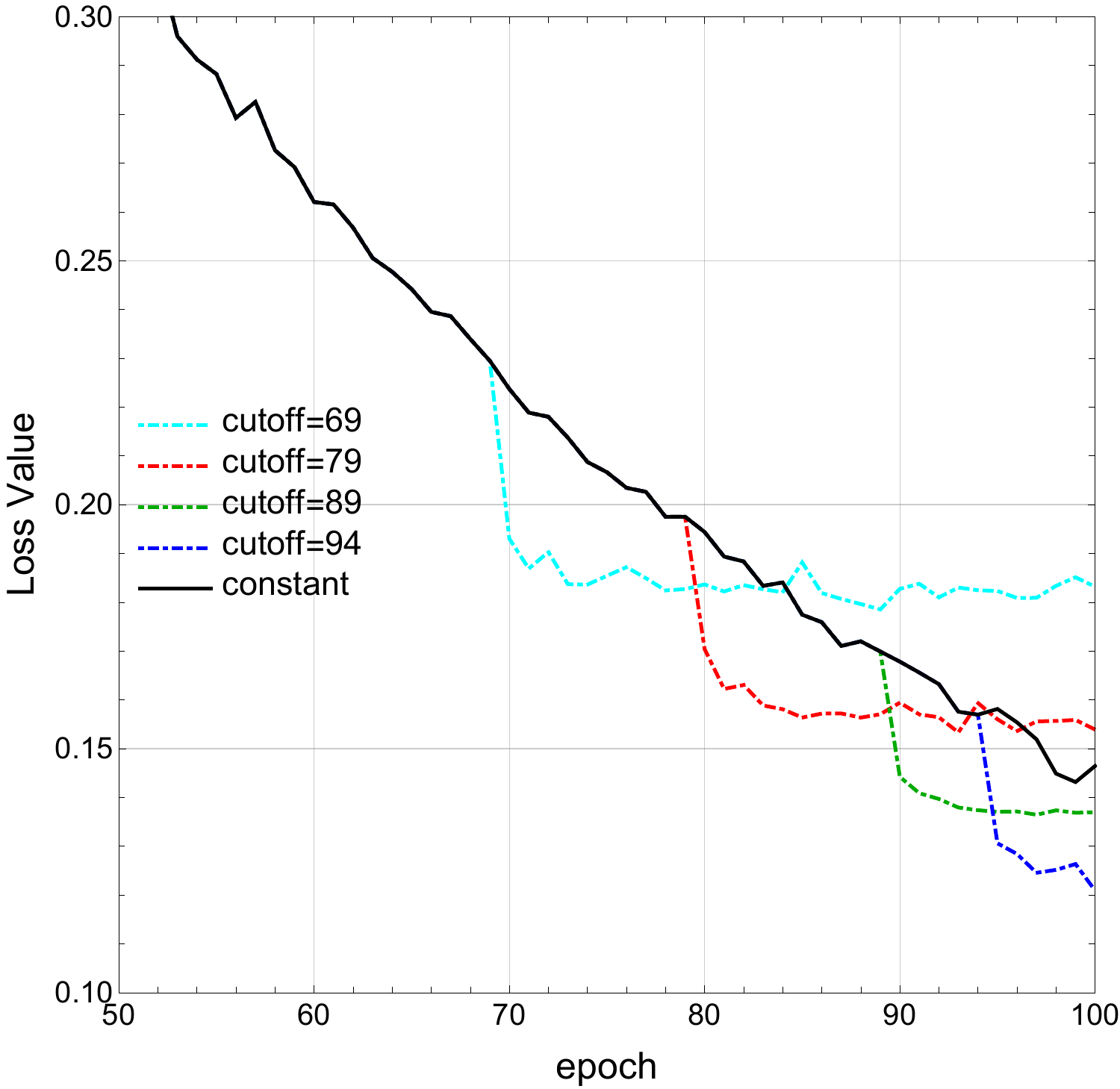}%
\caption{Training ResNet with a cooldown heuristic.}
\label{fig:ResNet}
\end{subfigure}
\end{figure*}

\section{Concluding remarks}
\label{sec:conclusion}

Our aim in this paper was to present a novel trajectory-based analysis of \eqref{eq:SGD} showing that, under minimal assumptions,
\begin{enumerate*}
[(\itshape i\hspace*{.5pt})]
\item
all of its limit points are stationary;
\item
it avoids strict saddle manifolds with probability $1$;
and
\item
it converges at a fast $\bigoh(1/\run)$ rate to regular minimizers.
\end{enumerate*}
This opens the door to many interesting directions \textendash\ from constrained/composite problems to adaptive gradient methods.
We defer these to the future.

\appendix
\numberwithin{equation}{section}		
\numberwithin{lemma}{section}		
\numberwithin{proposition}{section}		
\numberwithin{theorem}{section}		

\section{Convergence in continuous time}
\label{app:cont}

For completeness, we begin with a proof of the convergence of \eqref{eq:GD} under our blanket assumptions:

\begin{proposition}
[Gradient flow convergence]
\label{prop:conv-GD}
Under \cref{asm:reg,asm:sublvl}, every solution $\orbit{\ctime}$ of \eqref{eq:GD} converges to $\sols$.
\end{proposition}

\begin{proof}[Proof of \cref{prop:conv-GD}]
To begin, existence and uniqueness of (global) solutions to \eqref{eq:GD} follows readily from the Picard\textendash Lindelöf theorem \citep{Tes12} and \cref{asm:reg}.
With this point settled,
and given that the sublevel sets of $\obj$ are bounded (\cf \cref{asm:sublvl}), the fact that $\obj$ is non-increasing along the orbits of \eqref{eq:GD} shows that $\orbit{\ctime}$ converges to some compact invariant set $\cpt\subseteq\points$.

Suppose now that there exists a sequence of times $\ctime_{\run}$, $\run=\running$, such that $\orbit{\ctime_{\run}}$ converges to some \emph{non-critical} point $\olim \notin \crit(\obj) \equiv \sols$.
Letting $\const = \dnorm{\nabla\obj(\olim)}^{2} > 0$, there exists a neighborhood $\nhd$ of $\olim$ such that $\dnorm{\nabla\obj(\point)}^{2} \geq \const/2$ for all $\point\in\nhd$ (again, by \cref{asm:reg}).
Hence, by passing to a subsequence if necessary, we can assume without loss of generality that $\orbit{\ctime_{\run}} > \const/2$ for all $\run$.
Furthermore, by \cref{asm:reg} (which implies that $\norm{\nabla\obj(\point)} \leq \gbound$ for all $\point\in\points$) and the definition of \eqref{eq:GD}, we have
\begin{equation}
\norm{\orbit{\ctime_{\run} + \tau} - \orbit{\ctime_{\run}}}
	\leq \int_{\ctime_{\run}}^{\ctime_{\run} + \tau}
		\dnorm{\nabla\obj(\orbit{\ctimealt})}
		\dd\ctimealt
	\leq \vbound \tau
\end{equation}
for all $\tau > 0$.
Therefore, by picking $\tau$ sufficiently small, we can assume that $\orbit{\ctime} \in \nhd$ for all $\ctime \in [\ctime_{\run},\ctime_{\run} + \tau]$ and all $\run$ (recall here that $\orbit{\ctime_{\run}} \in \nhd$ for all $\run$).
Then, by the definition of $\nhd$, we readily get
\begin{align}
\obj(\orbit{\ctime_{\run} + \tau})
	&- \obj(\orbit{\ctime_{\run}})
	= \int_{\ctime_{\run}}^{\ctime_{\run} + \tau} \ddt \obj(\orbit{\ctimealt}) \dd\ctimealt
	= -\int_{\ctime_{\run}}^{\ctime_{\run} + \tau} \dnorm{\nabla\obj(\orbit{\ctimealt})}^{2} \dd\ctimealt
	\leq - \frac{\const\tau}{2},
\end{align}
and hence:
\begin{align}
\obj(\orbit{\ctime_{\run} + \tau}) - \obj(\orbit{\cstart})
	&= -\int_{0}^{\ctime_{\run} + \tau}
			\dnorm{\nabla\obj(\orbit{\ctimealt})}^{2}
		\dd\ctimealt
	\leq -\sum_{\runalt=\start}^{\run}
		\int_{\ctime_{\runalt}}^{\ctime_{\runalt} + \tau}
			\dnorm{\nabla\obj(\orbit{\ctimealt})}^{2}
		\dd\ctimealt
	\notag\\
	&\leq -\sum_{\runalt=\start}^{\run} \frac{\const \tau}{2}
	= -\frac{\run\const\tau}{2}
\end{align}
\ie $\lim_{\run\to\infty} \obj(\orbit{\ctime_{\run} + \tau}) = -\infty$, a contradiction.
Since $\orbit{\ctime}$ converges to a compact invariant set $\cpt$, we conclude that $\orbit{\ctime}$ in fact converges to the critical set $\sols$ of $\obj$.
\end{proof}

\section{Stability and boundedness of \acp{APT}}
\label{app:stability}

\subsection{Discrepancies between flows and \acp{APT}}
\label{sec:APT-bad}

Our first goal in this appendix is to provide a concrete example where \aclp{APT} and the underlying continuous-time flow exhibit qualitatively different behaviors in the long run.
To that end, consider the autonomous \acs{ODE}
\begin{equation}
\label{eq:ODE-bad}
\dorbit[1]{\ctime}
	= 1
	\qquad
\dorbit[2]{\ctime}
	= -\frac{\orbit[2]{\ctime}}{1 + \orbit[1]{\ctime}}
\end{equation}
which is a pseudo-gradient flow of the function $\obj(\point_{1},\point_{2}) = \point_{2}^{2}/2 - \point_{1}$.
The general solution of this system with initial condition $(0,b)$ at time $\ctime=\cstart$ is
\begin{equation}
\orbit{\ctime}
	= \parens*{\ctime,\frac{b}{1+\ctime}}.
\end{equation}
As a result, we have $\orbit[2]{\ctime}\to0$ as $t\to\infty$ from any initial condition (for a graphical illustration, see \cref{fig:APT}).


\begin{figure}[t]
\centering
\includegraphics[width=8cm]{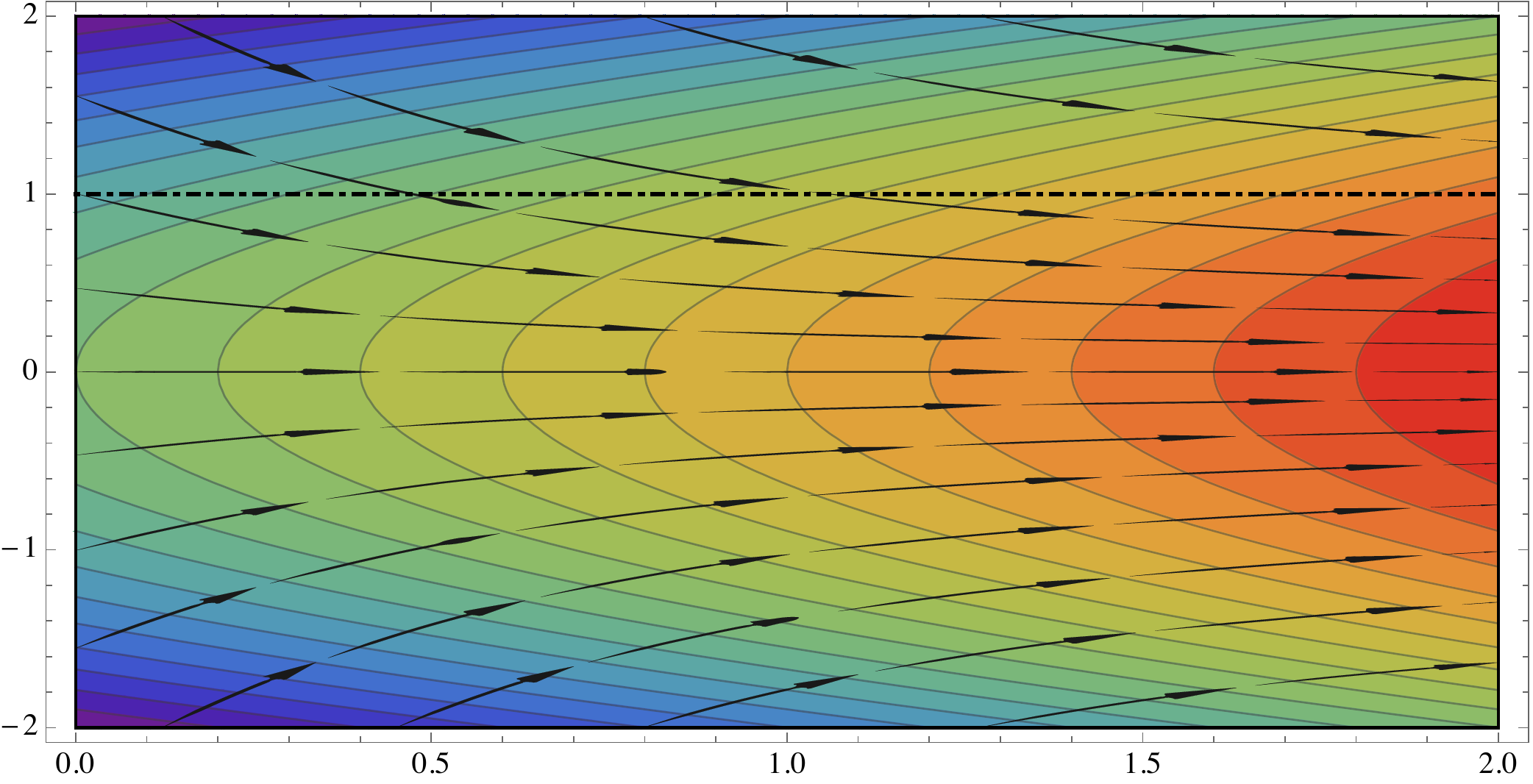}
\caption{Flowlines \vs \aclp{APT}:
the dashed black line is an \ac{APT} of the depicted gradient flow but it stays at a constant height ($\point_{2} = 1$), even though all flow lines converge to the $\point_{1}$-axis ($\point_{2}=0$).}
\label{fig:APT}
\end{figure}


On the other hand, as we show below, the ``constant height'' curve $\apt{\ctime} = (\ctime,1)$ is an \acl{APT} of \eqref{eq:ODE-bad}.
To show this, fix some accuracy threshold $\eps>0$ and a horizon $\horizon>0$.
Then, with respect to \Cref{def:APT}, it suffices to show that, for some sufficiently large $\ctime_{0}>0$ and all $h\in[0,\horizon]$, we have
\begin{equation}
\abs{1 - \orbit[2]{\ctime_{0} + h}}
	\leq \eps
\end{equation}
for the solution trajectory $\orbit{\ctime} = (\orbit[1]{\ctime},\orbit[2]{\ctime})$ that passes through the point $(\ctime_{0},1)$ at time $\ctime_{0}$.%
\footnote{That this is so is a consequence of the fact that the trajectories of \eqref{eq:ODE-bad} intersect the line $\point_{2}=1$ at a vanishing angle as $\ctime\to\infty$.
More precisely, if we show the statement in question for $\ctime_{0}$, it will also hold for all $\tau \geq \ctime_{0}$ by virtue of the monotonicity of the exponential function.}

Substituting in the general solution of \eqref{eq:ODE-bad} and backsolving, we readily obtain that this trajectory has
\begin{equation}
\orbit[2]{\ctime}
	= \frac{1 + \ctime_{0}}{1 + \ctime}.
\end{equation}
In turn, this implies that the maximal difference between $1$ and $\orbit[2]{\ctime}$ over a window of size $\horizon$ starting at $\ctime_{0}$ is
\begin{align}
\max_{0\leq h \leq \horizon} \abs{1 - \orbit[2]{\ctime_{0} + h}}
	&= \frac{1 + \ctime_{0} + \horizon}{1 + \ctime_{0}} - 1
	= \frac{\horizon}{1 + \ctime_{0}}
	\leq \eps
\end{align}
if $\ctime_{0}$ is chosen sufficiently large \textendash\ specifically, if $\ctime_{0} \geq \horizon / \eps - 1$.

Since $\eps$ is arbitrary, the above shows that the \ac{APT} condition \eqref{eq:APT} holds for all $\horizon>0$, \ie $\apt{\ctime}$ is an \ac{APT} of \eqref{eq:ODE-bad}.
On the other hand, we have $\lim_{\ctime\to\infty} \apt[2]{\ctime} = 1$, which is different than the limit of \emph{any} solution of \eqref{eq:ODE-bad}.

\subsection{Boundedness of \acp{APT}}
\label{sec:APT-bounded}

Our aim in the rest of this appendix will be to prove \cref{thm:stability}, which, for convenience, we restate below:

\stability*

To begin, we recall the basic \ac{APT} property of \eqref{eq:SGD}:

\propAPT*

The proof of \cref{prop:APT} follows by a tandem application of Propositions 4.1 and 4.2 of \citet{Ben99}, so we omit it;
instead, we focus directly on the proof of \cref{thm:stability}.
To that end, as we explained in the main body of the paper, the first part of our proof consists of showing that \eqref{eq:SGD} admits a subsequence converging to $\sols$, \ie that $\liminf_{\run\to\infty} \dist(\curr,\sols) = 0$:

\begin{lemma}
\label{lem:subsequence}
With assumptions as in \cref{thm:stability}, there exists a \textpar{possibly random} subsequence $\state_{\run_{\runalt}}$ of $\curr$ that converges to $\sols$;
formally, $\liminf_{\run\to\infty} \dist(\curr,\sols) = 0$ \as.
\end{lemma}

Before proving \cref{lem:subsequence}, we will require an intermediate result:

\begin{lemma}
\label{lem:smallgrad}
Let $\closed$ be a closed subset of $\points$ such that $\sols\cap\closed = \varnothing$.
Then, under \cref{asm:gradlvl}, $\inf_{\point\in\closed} \norm{\nabla\obj(\point)} > 0$.
\end{lemma}

\begin{proof}
Arguing by contradiction, assume there exists some sequence $\curr[\point] \in \closed$ such that $\norm{\nabla\obj(\curr[\point])} \to 0$ as $\run\to\infty$.
If $\curr[\point]$ admits a subsequence converging to some limit point $\olim\in\closed$, then, by continuity (recall that $\obj$ is assumed $C^{\vdim}$), we would also have $\norm{\nabla\obj(\olim)} = 0$.
In turn, this would imply $\olim\in\sols$, contradicting the assumption that $\closed$ is closed and disjoint from $\sols$.

Therefore, to prove our claim, it suffices to examine the case where $\curr[\point]$ has no convergent subsequence, \ie $\liminf_{\run\to\infty} \norm{\curr[\point]} = \infty$.
However, this would mean that the gradient sublevel set $\glvl_{\eps} = \setdef{\point\in\points}{\dnorm{\nabla\obj(\point)} \leq \eps}$ is unbounded for all $\eps>0$, in contradiction to \cref{asm:gradlvl}.
We conclude that $\liminf_{\run\to\infty} \norm{\nabla\obj(\curr[\point])} > 0$ for every sequence $\curr[\point]$ in $\closed$, \ie $\liminf_{\point\in\closed} \norm{\nabla\obj(\point)} > 0$.
\end{proof}

\begin{proof}[Proof of \cref{lem:subsequence}]
Assume ad absurdum that the event
\begin{equation}
\samples_{0}
	= \{\liminf\nolimits_{\run\to\infty} \dist(\curr,\sols) > 0\}
\end{equation}
occurs with positive probability.
By \cref{lem:smallgrad}, if $\liminf_{\run\to\infty} \dist(\curr,\sols) > 0$, we must also have $\liminf_{\run\to\infty} \norm{\nabla\obj(\curr)} > 0$ (since $\curr$ will eventually be contained in a closed set that is disjoint from $\sols$).
Therefore, fixing a realization $\curr$, $\run=\running$, of \eqref{eq:SGD} such that $\samples_{0}$ holds, there exists some (random) positive constant $\const>0$ with $\norm{\obj(\curr)}^{2} \geq \const$ for all sufficiently large $\run$;
without loss of generality, we may \textendash\ and will \textendash\ assume in the sequel that this actually holds for all $\run\geq\start$.

In view of all this, by the smoothness assumption for $\obj$ and the definition of \eqref{eq:SGD} we readily get:
\begin{align}
\obj(\next)
	= \obj(\curr - \curr[\step]\curr[\signal])
	&\leq \obj(\curr)
		- \curr[\step] \braket{\nabla\obj(\curr)}{\curr[\signal]}
		+ \frac{\smooth}{2} \curr[\step]^{2} \norm{\curr[\signal]}^{2}
	\notag\\
	&= \obj(\curr)
		- \curr[\step] \norm{\nabla\obj(\curr)}^{2}
		- \curr[\step] \braket{\nabla\obj(\curr)}{\curr[\noise]}
		+ \frac{\smooth}{2} \curr[\step]^{2} \norm{\curr[\signal]}^{2}
	\notag\\
	&\leq \obj(\curr)
		- \curr[\step] \const
		- \curr[\step] \curr[\snoise]
		+ \curr[\step]^{2} \smooth \norm{\nabla\obj(\curr)}^{2}
		+ \curr[\step]^{2} \smooth \norm{\curr[\noise]}^{2},
\end{align}
where we set $\curr[\snoise] = \braket{\nabla\obj(\curr)}{\curr[\noise]}$.
Therefore, setting $\curr[\obj] = \obj(\curr)$ and telescoping, we obtain
\begin{equation}
\label{eq:energy-bound1}
\next[\obj]
	\leq \init[\obj]
		- \curr[\proper]
			\bracks[\Bigg]{
				\const
				+ \underbrace{\frac{\sum_{\runalt=\start}^{\run} \iter[\step] \iter[\snoise]}{\curr[\proper]}}_{\curr[A]}
				- \smooth \underbrace{\frac{\sum_{\runalt=\start}^{\run} \iter[\step]^{2} \norm{\nabla\obj(\iter[\state])}^{2}}{\curr[\proper]}}_{\curr[B]}
				- \smooth \underbrace{\frac{\sum_{\runalt=\start}^{\run} \iter[\step]^{2} \norm{\iter[\noise]}^{2}}{\curr[\proper]}}_{\curr[C]}
				}, 
\end{equation}
where $\curr[\proper] = \sum_{\runalt=\start}^{\run} \iter[\step]$ is the ``elapsed time'' of $\curr$ as defined in \cref{sec:SA}.
We will proceed to show that all the summands in the brackets of \eqref{eq:energy-bound1} except the first converge to $0$;
since $\const>0$ and $\curr[\proper]\uparrow\infty$, this will show that $\lim_{\run\to\infty} \curr[\obj] = -\infty$, in direct contradiction to \cref{asm:sublvl}.

We carry out this plan term-by-term below:
\begin{enumerate}

\item
For the first term ($\curr[A]$), note that
\begin{equation}
\exof{\curr[\snoise] \given \curr[\filter]}
	= \exof{\braket{\nabla\obj(\curr)}{\curr[\noise]} \given \curr[\filter]}
	= \braket{\nabla\obj(\curr)}{\exof{\curr[\noise] \given \curr[\filter]}}
	= 0
\end{equation}
by \cref{asm:noise}.
This means that $\sum_{\runalt=\start}^{\run} \iter[\step] \iter[\snoise]$ is a zero-mean martingale, so, by the law of large numbers for martingale difference sequences \citep[Theorem 2.18]{HH80}, we have $\curr[\step]^{-1} \sum_{\runalt=\start}^{\run} \iter[\step] \iter[\snoise] \to 0$ \as on the event
\begin{equation}
\label{eq:cond-LLN}
\samples_{1}
	= \braces*{\sum_{\run=\start}^{\infty} \frac{\curr[\step]^{2}}{\curr[\proper]^{2}} \exof{\curr[\snoise]^{2} \given \curr[\filter]}
		< \infty}.
\end{equation}
However, by \cref{asm:reg,asm:noise}, we have
\usetagform{comment}
\begin{align}
\exof{\curr[\snoise]^{2} \given \curr[\filter]}
	&= \exof{\braket{\nabla\obj(\curr)}{\curr[\noise]}^{2} \given \curr[\filter]}
	\notag\\
	&\leq \gbound^{2} \exof{\norm{\curr[\noise]}^{2} \given \curr[\filter]}
	\tag{by \cref{asm:reg}}\\
	&\leq \gbound^{2} \exof{\norm{\curr[\noise]}^{\qexp} \given \curr[\filter]}^{2/\qexp}
	\tag{by Jensen}\\
	&\leq \gbound^{2} \noisevar
	\tag{by \cref{asm:noise}}
\end{align}
\usetagform{default}%
where, in the second-to-last line, we applied Jensen's inequality to the function $z \mapsto z^{\qexp/2}$ (recall here that $\qexp\geq2$).
Moreover, for all $\pexp\in(0,1]$, we have $\curr[\step]^{2} / \curr[\proper]^{2} = \tilde\bigoh(1/\run^{2})$, so $\sum_{\run=\start}^{\infty} \curr[\step]^{2} \big/ \curr[\proper]^{2} < \infty$.
Thus, going back to \eqref{eq:cond-LLN}, we conclude that $\curr[\step]^{-1} \sum_{\runalt=\start}^{\run} \iter[\step] \iter[\snoise] \to 0$ with probability $1$.

\item
For the second term ($\curr[B]$), simply note that $\norm{\nabla\obj(\curr)}^{2} \leq \gbound^{2}$, so we have:
\begin{equation}
\curr[B]
	= \frac{\sum_{\runalt=\start}^{\run} \iter[\step]^{2} \norm{\nabla\obj(\iter[\state])}^{2}}{\curr[\proper]}
	= \begin{cases}
		\bigoh(1/\run^{\pexp})
			&\quad
			\text{if $0<\pexp<1/2$},
			\\[\smallskipamount]
		\bigoh(\log\run/\sqrt{\run})
			&\quad
			\text{if $\pexp=1/2$},
			\\[\smallskipamount]
		\bigoh(1/\run^{1-\pexp})
			&\quad
			\text{if $1/2<\pexp<1$},
			\\[\smallskipamount]
		\bigoh(1/\log\run)
			&\quad
			\text{if $\pexp=1$}.
	\end{cases}
\end{equation}
Thus, from the above, we conclude that $\curr[B]\to0$.

\item
For the third term ($\curr[C]$), we will require a series of estimates.
First, with a fair degree of hindsight, let $\curr[Q] = \sum_{\runalt=\start}^{\run} \iter[\step]^{1+\qexp/2} \norm{\iter[\noise]}^{\qexp}$.
Noting that $\exof{\norm{\curr[\noise]}^{\qexp}} < \infty$ \as and $\exof{\curr[Q] \given \curr[\filter]} = \prev[Q] + \curr[\step]^{1+\qexp/2} \norm{\curr[\noise]}^{\qexp} \geq \prev[Q]$ for all $\run=\running$, we deduce that $\curr[Q]$ is a submartingale.
Furthermore, we have:
\begin{align}
\exof*{\sum_{\run=\start}^{\infty} \curr[\step]^{1+\qexp/2} \norm{\curr[\noise]}^{\qexp}}
	&\leq \sum_{\run=\start}^{\infty} \curr[\step]^{1+\qexp/2} \exof{\norm{\curr[\noise]}^{\qexp} }
	\leq \sum_{\run=\start}^{\infty} \curr[\step]^{1+\qexp/2} \noisepar^{\qexp}
	\notag\\
	&= \bigoh\parens*{\sum_{\run=\start}^{\infty} \run^{-\frac{\pexp(\qexp + 2)}{2}}}
	< \infty,
\end{align}
\ie $\curr[Q]$ is bounded in $L^{1}$ (recall that $\pexp > 2/(\qexp+2)$ by assumption).
Hence, by Doob's submartingale convergence theorem \citep[Theorem 2.1]{HH80}, it follows that $\curr[Q]$ converges \as to a random variable $Q_{\infty}$ with $\exof{Q_{\infty}} < \infty$ (and hence $Q_{\infty} < \infty$ with probability $1$ as well).

To proceed, we will need to consider two cases, depending on whether $\qexp=2$ or $\qexp>2$.
For the latter (which is more difficult), we will require the following variant of Hölder's inequality:
\begin{equation}
\parens*{\sum_{\runalt=\start}^{\run} \iter[\alpha] \iter[\beta]}^{\rexp}
	\leq
		\parens*{\sum_{\runalt=\start}^{\run} \iter[\alpha]^{\frac{\delta\rexp}{\rexp-1}}}^{\rexp-1}
		\sum_{\runalt=\start}^{\run} \iter[\alpha]^{(1-\delta)\rexp} \iter[\beta]^{\rexp},
\end{equation}
valid for all $\rexp>1$ and all $\delta\in(0,1)$.
Then, applying this inequality with $\iter[\alpha] = \iter[\step]^{2}$, $\iter[\beta] = \norm{\iter[\noise]}^{2}$, $\rexp = \qexp/2$ and $\delta = (\qexp-2)/(2\qexp)$, we obtain:
\begin{equation}
\parens*{\sum_{\runalt=\start}^{\run} \iter[\step]^{2} \norm{\iter[\noise]}^{2}}^{\qexp/2}
	\leq \parens*{ \sum_{\runalt=\start}^{\run} \iter[\step] }^{\qexp/2 - 1}
		\sum_{\runalt=\start}^{\run} \iter[\step]^{1+\qexp/2} \norm{\iter[\noise]}^{\qexp}
	= \curr[\proper]^{\qexp/2 - 1} \curr[Q],
\end{equation}
and hence:
\begin{equation}
\curr[C]
	= \frac{\sum_{\runalt=\start}^{\run} \iter[\step]^{2} \norm{\iter[\noise]}^{2}}{\curr[\proper]}
	\leq \frac{\curr[\proper]^{1 - 2 / \qexp} \curr[Q]^{2/\qexp}}{\curr[\proper]} 
	= \frac{\curr[Q]^{2/\qexp}}{\curr[\proper]^{2/\qexp}}.
\end{equation}
Since $\qexp>2$ and $\curr[Q]$ converges \as to $Q_{\infty}$, it follows that $\lim_{\run\to\infty} \curr[C] = 0$  with probability $1$ (since $\lim_{\run\to\infty} \curr[\proper] = \infty$ by our assumptions for $\curr[\step]$).
Finally, if $\qexp=2$, we have $\curr[C] = \curr[Q] / \curr[\proper]$ by definition, so we get $\curr[C] \to 0$ \as directly.
\end{enumerate}

Putting together all of the above, we get $\curr[A] + \curr[B] + \curr[C] \to 0$ with probability $1$, and hence, with probability $1$ conditioned on $\samples_{0}$ (since $\probof{\samples_{0}} > 0$).
This means that, for sufficiently large $\run$, we have
\begin{equation}
\next[\obj]
	\leq \init[\obj]
		- \curr[\proper] (\const/2)
\end{equation}
which, together with the fact that $\lim_{\run\to\infty} \curr[\proper] = \infty$, implies that $\lim_{\run\to\infty} \obj(\curr) = -\infty$.
This contradicts \cref{asm:sublvl} and completes our proof.
\end{proof}

We now move on to the deterministic elements of the proof of \cref{thm:stability}.
To that end, let
\begin{equation}
\label{eq:critval}
\critval
	= \max_{\point\in\sols} \obj(\point)
\end{equation}
denote the maximum value of $\obj$ over its critical set, and let
\begin{equation}
\label{eq:critnhd}
\critnhd_{\eps}
	= \flvl_{\critval+\eps}
	= \setdef{\point\in\points}{\obj(\point) \leq \critval + \eps}
\end{equation}
denote the $(\critval+\eps)$-sublevel set of $\obj$.
We then have the following ``uniform decrease'' estimate:

\begin{lemma}
\label{lem:decrease}
Fix some $\eps>0$.
Under \cref{asm:reg,asm:sublvl,asm:gradlvl}, there exists some $\tau \equiv \tau(\eps)$ such that, for all $\point\in\points$, we have
\begin{enumerate*}
[\upshape(\itshape i\hspace*{.5pt}\upshape)]
\item
$\obj(\flow{\tau}{\point}) \leq \obj(\point) - \eps$;
or
\item
$\flow{\tau}{\point} \in \critnhd_{\eps}$.
\end{enumerate*}
\end{lemma}

\begin{proof}
By \cref{lem:smallgrad}, there exists some positive constant $\const>0$ such that $\norm{\nabla\obj(\point)}^{2} \geq \const > 0$ for all $\point \in \points \setminus \critnhd_{\eps}$.
Then, with $d\obj/dt = - \norm{\nabla{\obj(\orbit{\ctime})}}^{2}$, if we let $\tau_{\point} = \inf\setdef{\ctime\geq0}{\flow{\ctime}{\point} \in \critnhd_{\eps}}$, we get:
\begin{equation}
\label{eq:decrease}
\obj(\flow{\ctime}{\point})
	= \obj(\point)
		- \int_{0}^{\ctime} \norm{\nabla\obj(\orbit{\ctimealt})}^{2} \dd\ctimealt
	\leq \obj(\point)
		- \const \ctime
	\quad
	\text{for all $\ctime\in[0,\tau_{\point}]$.}
\end{equation}
Accordingly, letting $\tau = \eps/\const$, we may consider the following two case:
\begin{enumerate}
\item
If $\tau_{\point} \geq \tau$, applying \eqref{eq:decrease} for $\ctime = \tau$ yields $\obj(\flow{\tau}{\point}) \leq \obj(\point) - \eps$.
\item
Otherwise, if $\tau_{\point} < \tau$, we have $\obj(\flow{\tau}{\point}) \leq \obj(\flow{\tau_{\point}}{\point}) \leq \critval + \eps$, implying in particular that $\flow{\tau}{\point} \in \critnhd_{\eps}$.
\end{enumerate}
Our claim then follows by combining the two cases above.
\end{proof}

Finally, we establish below the required comparison bound between an \ac{APT} of \eqref{eq:GD} and its solution trajectories:

\begin{lemma}
\label{lem:flowAPT}
Fix some $\eps,\delta>0$.
Then, with assumptions and notation as in \cref{lem:decrease}, there exists some $\ctime_{0} \equiv \ctime_{0}(\delta,\eps)$ such that, for all $\ctime \geq \ctime_{0}$ and all $h \in [0,\tau]$, we have:
\begin{equation}
\label{eq:flowAPT}
\obj(\apt{\ctime + h})
	\leq \obj(\flow{h}{\apt{\ctime}})
		+ \lips \delta
		+ \tfrac{1}{2} \smooth \delta^{2}.
\end{equation}
\end{lemma}

\begin{proof}
By the definition of an \ac{APT}, there exists some $\ctime_{0} \equiv \ctime_{0}(\delta,\eps)$ such that
\begin{equation}
\label{eq:APT-delta}
\sup_{0\leq h\leq \tau} \norm{\apt{\ctime + h} - \flow{h}{\apt{\ctime}}}
	\leq \delta
\end{equation}
for all $\ctime \geq \ctime_{0}$.
Hence, for all $\ctime\geq\ctime_{0}$ and all $h\in[0,\tau]$, we have
\begin{align}
\obj(\apt{\ctime+h})
	&= \obj(\flow{h}{\apt{\ctime}} + \apt{\ctime+h} - \flow{h}{\apt{\ctime}})
	\notag\\
	&\leq \obj(\flow{h}{\apt{\ctime}})
		+ \braket{\nabla\obj(\flow{h}{\apt{\ctime}})}{\apt{\ctime+h} - \flow{h}{\apt{\ctime}}}
	\notag\\
	&\qquad
		+ \frac{\smooth}{2} \norm{\apt{\ctime+h} - \flow{h}{\apt{\ctime}}}^{2}
	\notag\\
	&\leq \obj(\flow{h}{\apt{\ctime}})
		+ \gbound \norm{\apt{\ctime+h} - \flow{h}{\apt{\ctime}}}
		+ \frac{\smooth}{2} \norm{\apt{\ctime+h} - \flow{h}{\apt{\ctime}}}^{2}
	\notag\\
	&\leq \obj(\flow{h}{\apt{\ctime}})
		+ \gbound\delta
		+ \frac{\smooth}{2} \delta^{2},
\end{align}
as claimed.
\end{proof}
With all this in hand, we are finally in a position to formally prove \cref{thm:stability}.

\begin{proof}[Proof of \cref{thm:stability}]
We will prove the stronger statement that, with probability $1$, $\apt{\ctime}$ converges to the sublevel set $\flvl_{\critval} = \setdef{\point\in\points}{\obj(\point) \leq \critval}$ with $\critval$ defined as in \eqref{eq:critval}.
Since the sublevel sets of $\obj$ are bounded, convergence to $\flvl_{\critval}$ suffices.

To prove this claim, fix some $\eps>0$ and let $\apt{\ctime}$ be the affine interpolation of the sequence of iterates $\curr$ generated by \eqref{eq:SGD}.
Under the stated assumptions, \cref{prop:APT} guarantees that $\apt{\ctime}$ is an \ac{APT} of \eqref{eq:GD} with probability $1$.
Moreover, again with probability $1$, \cref{lem:subsequence} guarantees the existence of some (possibly random) $\ctime_{1}$ such that $\apt{\ctime_{1}} \in \critnhd_{2\eps}$.
To streamline the analysis to come, we will condition our statements on the intersection of these two events (which still occurs with probability $1$), and we will argue trajectory-wise.

Moving forward, \cref{lem:decrease} guarantees the existence of some $\tau \equiv \tau(\eps)$ such that $\obj(\flow{\tau}{\point}) \leq \obj(\point) - \eps$ or $\flow{\tau}{\point} \in \critnhd_{\eps}$ for all $\point\in\points$.
Fixing this $\tau$ and taking $\delta>0$ such that $\gbound\delta + \smooth\delta^{2}/2 < \eps$, \cref{lem:flowAPT} further implies that there exists some $\ctime_{0}$ such that \eqref{eq:flowAPT} holds for all $\ctime\geq\ctime_{0}$ and all $h\in[0,\tau]$.
Note also that, without loss of generality, we can assume that $\ctime_{1} > \ctime_{0}$;
otherwise, if this is not the case, it suffices to wait for the first instance $\run$ such that $\curr\in\critnhd_{2\eps}$ and $\curr[\proper] \geq \ctime_{0}$ (by \cref{lem:subsequence}, this occurs with probability $1$).

Combining all of the above, we have
\begin{enumerate*}
[(\itshape i\hspace*{.5pt}\upshape)]
\item
$\apt{\ctime_{1}} \in \critnhd_{2\eps}$;
and
\item
$\obj(\apt{\ctime+h}) \leq \obj(\flow{h}{\apt{\ctime}}) + \eps$ for all $\ctime\geq\ctime_{1}$ and all $h\in[0,\tau]$.
\end{enumerate*}
Since $\obj(\flow{\ctime}{\point}) \leq \obj(\point)$ for all $\ctime\geq0$, this further implies that
\begin{equation}
\obj(\apt{\ctime+h})
	\leq \obj(\apt{\ctime})
		+ \eps
\end{equation}
for all $h\in[0,\tau]$.
We thus get
\begin{equation}
\obj(\apt{\ctime})
	\leq \obj(\apt{\ctime_{1}})
		+ \eps \leq \critval + 3\eps
\end{equation}
for all $\ctime\in[\ctime_{1},\ctime_{1}+\tau]$.
Moreover, since $\apt{\ctime_{1}} \in \critnhd_{2\eps}$, \cref{lem:decrease} also gives $\flow{\tau}{\apt{\ctime_{1}}} \in \critnhd_{\eps}$ because the two conditions of the lemma coincide if $\point\in\critnhd_{2\eps}$.
As a result, we finally obtain
\begin{equation}
\obj(\apt{\ctime_{1}+\tau})
	\leq \obj(\flow{\tau}{\apt{\ctime_{1}}}) + \eps
	\leq \critval + \eps + \eps
	= \critval + 2\eps,
\end{equation}
\ie $\apt{\ctime_{1} + \tau} \in \critnhd_{2\eps}$.

From the above, we conclude that
\begin{enumerate*}
[(\itshape i\hspace*{.5pt}\upshape)]
\item
$\apt{\ctime} \in \critnhd_{3\eps}$ for all $\ctime\in[\ctime_{1},\ctime_{1}+\tau]$;
and, in particular,
\item
$\apt{\ctime_{1} + \tau} \in \critnhd_{2\eps}$.
\end{enumerate*}
Proceeding inductively, we get $\apt{\ctime} \in \critnhd_{3\eps}$ for all $\ctime\in[\ctime_{1}+(\runalt-1)\tau,\ctime_{1}+\runalt\tau]$, $\runalt=\running$, \ie $\apt{\ctime} \in \critnhd_{3\eps}$ for all $\ctime\geq\ctime_{1}$.
Since $\eps>0$ is arbitrary, this means that $\apt{\ctime}$ converges to $\critnhd_{0} \equiv \flvl_{\critval}$ as claimed.
\end{proof}

\section{Avoidance analysis}
\label{app:avoidance}

As we stated in the main body of the paper, the proof of \cref{thm:avoidance} will require two different threads of arguments:
\begin{enumerate*}
[\itshape a\upshape)]
\item
a series of probabilistic estimates to show that a certain class of stochastic processes avoids zero;
and
\item
the construction of a suitable (average) Lyapunov function that grows exponentially along the unstable directions of a strict saddle manifold.
\end{enumerate*}

\subsection{Probabilistic estimates}
\label{sec:Pemantle}


The probabilistic estimates that we will need date back to \citet{Pem90} and concern a class of stochastic processes defined as follows:
let $\curr[Y]$, $\run=\running$, be a sequence of $\curr[\filter]$-measurable random variables,
let $\curr[\lyap] = \sum_{\runalt=\start}^{\run} \iter[Y]$,
and assume that
\begin{equation}
\label{eq:energy-unst}
\exof{\next[\lyap]^{2} - \curr[\lyap]^{2} \given \curr[\filter]}
	\geq C/\run^{2\pexp}
	\quad
	\text{for some $C>0$ and all $\run=\running$}
\end{equation}
In the above,
$\curr[\lyap]$ will play the role of a ``distance measure'' from $\saddles$.
Informally, the requirement \eqref{eq:energy-unst} posits that $\curr[\lyap]$ increases in ``root mean square'' by $\Theta(\curr[\step])$ where $\curr[\step] \propto 1/\run^{\pexp}$ is the step-size of \eqref{eq:SGD};
constructing such a process will be the topic of the geometric constructions of the next section.
For now, we state without proof a number of conditions guaranteeing that the process $\curr[\lyap]$ cannot converge to $0$:

\begin{lemma}[$0 < \pexp \leq 1/2$; \citealp{BH96}, Lemma 4.2]
\label{lem:Pem-large}
Suppose that \eqref{eq:energy-unst} holds for some $\pexp\in(0,1/2]$.
Then, $\probof{\lim_{\run\to\infty} \curr[\lyap] = 0} = 0$.
\end{lemma}

\begin{lemma}[$1/2 < \pexp \leq 1$: \citealp{Pem92}, Lemma 5.5]
\label{lem:Pem-small}
Suppose that \eqref{eq:energy-unst} holds for some $\pexp\in(1/2,1]$.
Assume further that there exist constants $a,b>0$ such that, for all $\run=\running$, we have:
\begin{enumerate}
\item
$\abs{\curr[Y]} \leq a/\run^{\pexp}$ with probability $1$.
\item
$\oneof{\curr[\lyap] > b/\run^{\pexp}} \exof{\next[Y] \given \curr[\filter]} \geq 0$ with probability $1$.
\end{enumerate}
Then, $\probof{\lim_{\run\to\infty} \curr[\lyap] = 0} = 0$.
\end{lemma}

A first version of \cref{lem:Pem-small} was originally proven by \citet{Pem90} for the special case $\pexp=1$ but the proof techniques are similar for all $1/2 < \pexp\leq1$;
for a more general estimate (which we will not need here), see \citet[Lemma~9.6]{Ben99}.

\subsection{Center manifold theory and geometric constructions}
\label{sec:energy}

We now proceed with the construction of a suitable Lyapunov function that will allow us to apply \cref{lem:Pem-large,lem:Pem-small}.
This construction follows \citet{BH95} and \citet{Ben99} and relies crucially on center manifold theory;
for a general introduction to the topic, we refer the reader to \citet{Lee03,Shu87}, and \citet{Rob12}.

To begin, let $\saddles$ be a strict saddle manifold as defined in \cref{sec:convergence}.
Then, for all $\saddle\in\saddles$, we define the \emph{center}, \emph{stable} and \emph{unstable} directions of $\sol$ to be respectively the 
eigenspaces of $\hmat(\saddle) = \nabla^{2}\obj(\saddle)$ corresponding to zero, positive and negative eigenvalues thereof, \ie
\begin{subequations}
\begin{alignat}{2}
\zerdirs{\saddle}
	&= \setdef{\vvec\in\R^{\vdim}}{\hmat(\saddle)\vvec = 0}
		= \ker\hmat(\saddle),
	&\hspace{4em}
	&[\text{central directions}]
	\\[\smallskipamount]
\posdirs{\saddle}
	&= \setdef{\vvec\in\R^{\vdim}}{\hmat(\saddle)\vvec = \lambda \vvec \; \text{for some $\lambda>0$}}
	&\hspace{4em}
	&[\text{stable directions}]
	\\[\smallskipamount]
\negdirs{\saddle}
	&= \setdef{\vvec\in\R^{\vdim}}{\hmat(\saddle)\vvec = \lambda \vvec\; \text{for some $\lambda<0$}}
	&\hspace{4em}
	&[\text{unstable directions}]
\end{alignat}
\end{subequations}

The reason for this terminology is that $\hmat(\point) = \Jac(\nabla\obj(\point))$, so these subspaces correspond to directions that are respectively neutral (or \emph{slow}), attracting, and repelling under \eqref{eq:GD}.
More preciselly, by the center manifold theorem \cite{Shu87,Rob12}, there exists a neighborhood $\nhd$ of $\saddles$ and a submanifold $\mfld$ of $\vecspace$, called the \emph{center stable manifold} of $\saddles$, and satisfying the following:
\begin{enumerate*}
[\itshape a\upshape)]
\item
$\mfld$ is \emph{locally invariant} under $\flowmap$, \ie there exists some positive $\ctime_{0} > 0$ such that $\flow{\ctime}{\nhd\cup\mfld} \subseteq \mfld$ for all $\ctime\geq\ctime_{0}$;
and
\item
$\vecspace = \tangent{\saddle}{\mfld} \oplus \negdirs{\saddle}$ for all $\saddle\in\saddles$, where $\tangent{\saddle}{\mfld}$ denotes the tangent space to $\mfld$ at $\saddle$.
\end{enumerate*}
In view of this:
\begin{enumerate*}
[\itshape a\upshape)]
\item
perturbations along central directions are tangent to $\mfld$ and are thus expected to evolve ``along'' $\mfld$ under \eqref{eq:GD};
\item
stable perturbations along $\posdirs{\saddle}$ will converge along $\mfld$ to $\saddles$ under \eqref{eq:GD};
and
\item
unstable perturbations along $\negdirs{\saddle}$ are transverse to $\mfld$ and may escape.
\end{enumerate*}

A key property of $\mfld$ is that any globally bounded orbit of \eqref{eq:GD} which is contained in a sufficiently small neighborhood of $\saddle\in\saddles$ must be entirely contained in $\mfld$ \citep{Shu87}.
Moreover, by the non-minimality assumption for $\saddles$, it follows that $\vdim_{u} \equiv \dim\negdirs{\saddle} \geq 1$, so the dimension of $\mfld$ is at most $\vdim-1$.
This suggests that perturbations along any direction that is transverse to $\mfld$ will be repelled under \eqref{eq:GD};
we make this statement precise in the lemma below.

\begin{lemma}
\label{lem:unstable}
Let $\gen{\ctime}{\point} = \nabla_{\point} \flow{\ctime}{\saddle}$ denote the infinitesimal generator of the flow of \eqref{eq:GD}.
Then:
\begin{enumerate}
\item
The unstable subspaces $\negdirs{\saddle}$ are invariant under \eqref{eq:GD};
specifically, $\gen{\ctime}{\saddle} \negdirs{\saddle} = \negdirs{\saddle}$ for all $\ctime\geq0$ and all $\saddle\in\saddles$.
\item
There exists a positive constant $\const>0$ such that, for all $\saddle\in\saddles$, $\wvec\in\negdirs{\saddle}$ and $\ctime\geq0$, we have
\begin{equation}
\label{eq:unstable-exp}
\norm{\gen{\ctime}{\saddle} \wvec}
	\geq e^{\const\ctime} \norm{\wvec}. 
\end{equation}
\end{enumerate}
\end{lemma}
\smallskip

\begin{remark}
In the above (and what follows), we write $AW$ for the image of a vector space $W$ under a linear operator $A$.
Specifically, if $A\from V \to V'$ is a linear operator between two vector spaces $V$ and $V'$, and if $W\leq V$ is a subpace of $V$, we let $AW \equiv \im_{A}(W) = \setdef{Aw}{w\in W}$.
We also treat linear operators and matrices interchangeably.
\end{remark}

\begin{remark}
The proof of \cref{lem:unstable} (and, in fact, all of our analysis in this section) does not require the uniformity condition $\min \lambda_{+}(\hmat(\saddle)) \geq c_{+}$ for the Hessian's positive eigenvalues (if such eigenvalues exist).
We only make it to simplify the presentation and avoid cases where the dimension of $\posdirs{\saddle}$ may change;
in that case, it would be sufficient to work with a subset of $\saddles$ over which this does not occur.
\end{remark}

In words, \cref{lem:unstable} states that
\begin{enumerate*}
[\itshape a\upshape)]
\item
the unstable directions along $\saddles$ are consistent with the flow of \eqref{eq:GD};
and
\item
perturbations along unstable directions are repelled from $\saddles$ at a geometric rate.
\end{enumerate*}
The proof is as follows:

\begin{proof}[Proof of \cref{lem:unstable}]
Recall first that, for all $\ctime\geq0$ and all $\point\in\points$, we have $\gen{\ctime}{\point} = \nabla_{\point} \flow{\ctime}{\point} = \exp(\ctime\Jac(-\nabla\obj(\point)))= \exp(-\ctime\hmat(\point))$.
Therefore, since $\saddles$ consists entirely of stationary points of \eqref{eq:GD}, we readily get
\usetagform{comment}
\begin{align}
\gen{\ctime}{\saddle} \negdirs{\saddle}
	&= e^{-\ctime\hmat(\saddle)} \negdirs{\saddle}
	\notag\\
	&= \sum_{k=0}^{\infty} \frac{(-\ctime)^{k}}{k!} \hmat(\saddle)^{k} \negdirs{\saddle}
	= \sum_{k=0}^{\infty} \frac{(-\ctime)^{k}}{k!} \negdirs{\saddle}
	\tag{because $\hmat(\saddle)\negdirs{\saddle} = \negdirs{\saddle}$}
	\\
	&= e^{-\ctime} \negdirs{\saddle}
	= \negdirs{\saddle},
\end{align}
\usetagform{default}
so our first claim follows.

For our second claim,
let $\setdef{\unitvec_{i}}{i=1,\dotsc,\vdim}$ be an orthnormal set of eigenvectors of $\hmat(\saddle)$.%
\footnote{That such a set exists follows from the fact that $\hmat(\saddle)$ is symmetric.}
Moreover, let $\lambda_{i} \equiv \lambda_{i}(\saddle) < 0$ be the eigenvalue of $\hmat(\saddle)$ corresponding to $\unitvec_{i}$, and assume without loss of generality that the indexing labels $i=1,\dotsc,\vdim$ have been chosen in ascending eigenvalue order, \ie $\lambda_{1} \leq \dotsi \leq \lambda_{\vdim}$.
It then follows that $\setdef{\unitvec_{i}}{i=1,\dotsc,\vdim_{u} \equiv \dim\negdirs{\saddle}}$ is an orthonormal basis of $\negdirs{\saddle}$ consisting entirely of eigenvectors of $\hmat(\saddle)$.
Thus, writing $\wvec = \sum_{i} \wvec_{i} \unitvec_{i}$ for a given vector $\wvec\in\negdirs{\saddle}$, we have:
\begin{equation}
\gen{\ctime}{\saddle} \wvec
	= e^{-\ctime\hmat(\saddle)} \wvec
	= \sum_{i=1}^{\vdim_{u}} \wvec_{i} e^{-\ctime\hmat(\saddle)} \unitvec_{i}
	= \sum_{i=1}^{\vdim_{u}} \wvec_{i} e^{-\ctime\lambda_{i}} \unitvec_{i},
\end{equation}
where, in the last step, we used the fact that $\unitvec_{i}$ is an eigenvector of $\hmat(\saddle)$ with eigenvalue $\lambda_{i}$ (and hence, also of $e^{-\ctime\hmat(\saddle)}$ with eigenvalue $e^{-t\lambda_{i}}$).
Therefore, by orthonormality, we obtain: 
\begin{equation}
\norm{\gen{\ctime}{\saddle}\wvec}^{2}
	= \sum_{i=1}^{\vdim_{u}} e^{-2\ctime\lambda_{i}} \wvec_{i}^{2}
	\geq e^{2\const_{-}\ctime} \norm{\wvec}^{2},
\end{equation}
where $\const_{-}>0$ is defined in \cref{sec:avoidance}.
\end{proof}

To proceed, we will need to define a suitable ``projector'' from neighborhoods of $\saddles$ to $\mfld$.
To carry out this construction, consider the vector bundle
\begin{equation}
\negdirs{\saddles}
	\equiv \setdef{(\saddle,\wvec)}{\saddle\in\saddles,\wvec\in\negdirs{\saddle}}
\end{equation}
of the unstable directions of \eqref{eq:GD} over $\saddles$.
Since each $\negdirs{\saddle}$ is a subspace of $\points$, we can view $\negdirs{\saddles}$ as a map from $\saddles$ to the Grassmannian $\Grass(\vdim_{u},\vdim)$ of $\vdim_{u}$-dimensional spaces of $\R^{\vdim}$.
By the Whitney embedding theorem \citep{Lee03}, $\Grass(\vdim_{u},\vdim)$ can be embedded as a $\vdim_{u} \times (\vdim-\vdim_{u})$-dimensional submanifold of $\R^{2\vdim_{u}(\vdim-\vdim_{u})}$;
as such, $\negdirs{\saddles}$ may be seen as a map $\saddles \to \R^{2\vdim_{u}(\vdim-\vdim_{u})}$ with values in $\Grass(\vdim_{u},\vdim) \injects \R^{2\vdim_{u}(\vdim-\vdim_{u})}$.
Since $\saddles$ is closed (as a connected component of $\sols$), the Tietze extension theorem \citep{Arm83} further implies that this map admits a continuous extension $\pi\from\R^{\vdim}\to\R^{2\vdim_{u}(\vdim-\vdim_{u})}$ to all of $\R^{\vdim}$.
By mollifying this map with an approximate identity supported on $\saddles$, we can further assume that this extension is smooth in a neighborhood of $\saddles$.
Moreover, by standard results in differential topology \citep[Chap.~4]{Hir76}, there exists a smooth retraction of a neighborhood of $\Grass(\vdim_{u},\vdim)$ onto $\Grass(\vdim_{u},\vdim)$ in $\R^{2\vdim_{u}(\vdim-\vdim_{u})}$.
Hence, by composing $\pi$ with this retraction, we finally obtain a smooth vector bundle
\begin{equation}
\negdirs{\nhd}
	\equiv \setdef{(\point,\wvec)}{\point\in\nhd,\wvec\in\negdirs{\point}}
\end{equation}
which, by construction, coincides with $\negdirs{\saddles}$ over $\saddles$ (explaining the slight abuse of notation).

By taking a smaller neighborhood if necessary, we may assume that $\nhd$ is compact and coincides with the one in the definition of $\mfld$, \ie $\flow{\ctime}{\nhd\cap\mfld} \subseteq \mfld$ for small enough $\ctime$.
We may now construct a ``projector'' from a (potentially smaller) neighborhood of $\mfld$ to $\mfld$ as follows:
First, consider the simple vector addition mapping $Q\from\negdirs{\nhd} \to \points \equiv \R^{\vdim}$ sending $(\point,\wvec) \in \negdirs{\nhd} \mapsto \point+\wvec \in \R^{\vdim}$.
Clearly, the zero section $(\point,0)$ of $\negdirs{\nhd}$ is mapped diffeomorphically to $\nhd$ so, by the inverse function theorem \citep{Lee03}, it follows that $Q$ is a local diffeomorphism.
Thus, letting $\alt\nhd$ be a neighborhood of $\mfld$ over which $Q$ is a diffeomorphism, and letting $\nhd_{0} = Q(\alt\nhd)$, we get a map $\Pi\from\nhd_{0} \to \mfld$ such that
\begin{equation}
\label{eq:projector}
\Pi(\wpoint)
	= \point
	\iff
Q(\point,\wvec)
	= \point + \wvec
	= \wpoint
\end{equation}

The reason for this sophisticated construction (as opposed to \eg taking a Euclidean projection from $\nhd_{0}$ to $\mfld$) is that $\Pi$ respects the unstable directions of $\saddles$ under \eqref{eq:GD}.
More precisely, we have:

\begin{lemma}
\label{lem:projector}
For $\point\in\nhd$, let $\proxmap_{\point} \from \tangent{\point}{\mfld} \oplus \negdirs{\point} \to \tangent{\point}{\mfld}$ denote the projection
\begin{equation}
\label{eq:proj-fiber}
\underset{\underset{\scriptstyle \tangent{\point}{\mfld}\oplus \negdirs{\point}\quad}{\mathbin{\rotatebox[origin=c]{90}{$\scriptstyle\ni$}}}}{\tvec+\wvec}
	\mapsto \prox{\point}{\tvec + \wvec}
	=
	\underset{\underset{\scriptstyle \tangent{\point}{\mfld}}{\mathbin{\rotatebox[origin=c]{90}{$\scriptstyle\ni$}}}}{\tvec}
\end{equation}
Then, for all $\point \in \nhd_{0} \cap \mfld$, we have $\Jac(\Pi(\point)) = \proxmap_{\point}$.
\end{lemma}

\begin{proof}
Let $\wpoint(\ctime)$, $\ctime\in(-1,1)$ be a smooth curve on $\nhd_{0}$ going through $\point = \wpoint(0)\in\mfld$ at time $\ctime=0$, and let $\point(\ctime) = \Pi(\wpoint(\ctime))$ so $\wpoint(\ctime) = \point(\ctime) + \psi(\ctime)$ for some smooth $\psi(\ctime) \in \negdirs{\point(\ctime)}$.
By differentiating, we get $\dot\wpoint(0) = \dot\point(0) + \dot\psi(0)$;
since $\point(\ctime) \in \mfld$ and $\psi(\ctime) \in \negdirs{\point(\ctime)}$ for all $\ctime$, we readily get $\dot\point(0) \in \tangent{\point(0)}{\mfld}$ and $\dot\psi(0) \in \negdirs{\point(0)}$.
Letting $\tvec = \dot\point(0)$ and $\wvec = \dot\psi(0)$, this shows that the pushforward of $\dot\wpoint(0) = \tvec + \wvec$ under $\Pi$ at $\point$ is $\mathrm{D}\Pi_{\point}(\tvec + \wvec) \equiv \Jac(\Pi(\point))(\tvec + \wvec) = \tvec = \prox{\point}{\tvec + \wvec}$.
With $\wpoint(\ctime)$ arbitrary, our claim follows.
\end{proof}

We are finally in a position to define a ``potential function'' on $\nhd_{0}$ as
\begin{equation}
\label{eq:pot}
\prelyap(\wpoint)
	= \norm{\Pi(\wpoint) - \wpoint}
\end{equation}
\ie as the (normed) distance of $\wpoint\in\nhd_{0}$ from its vector projection $\Pi(\wpoint)$ on $\mfld$ along the unstable directions of \eqref{eq:GD}.
By construction, we have
\begin{equation}
\label{eq:pot-posdef}
\prelyap(\wpoint)
	\geq0
	\quad
	\text{with equality if and only if $\wpoint \in \mfld \cap \nhd_{0}$}.
\end{equation}
Coupling \eqref{eq:pot-posdef} with \cref{lem:unstable,lem:projector}, we see that $\pot$ satisfies the requirements of \citet[Proposition 9.5]{Ben99}, which, when adapted to our setting, provides the following:

\begin{proposition}[\citealp{Ben99}]
\label{prop:energy}
There exists a compact neighborhood $\nhd_{\saddles}$ of $\saddles$,
a positive constant $\beta>0$,
and
a time horizon $\tau>0$ such that the energy function
\begin{equation}
\label{eq:energy-saddle}
\lyap(\point)
	= \int_{0}^{\tau} \prelyap(\flow{-\ctime}{\point}) \dd\ctime
	\qquad
	\point\in\nhd_{\saddles},
\end{equation}
enjoys the following properties:
\begin{enumerate}

\item
For all $\point\in\nhd_{\saddles}$, $\lyap$ has a Lipschitz continuous and positively homogeneous right derivative $\nabla^{+} \lyap(\point)$;%
\footnote{Recall here that a function $\phi$ has a right derivative when the limit $\nabla^{+}\phi(\point)[\vvec] \equiv \lim_{\ctime\to0^{+}} \bracks{\phi(\point+\ctime\vvec) - \phi(\point)} / \ctime$ exists for all $\vvec\in\R^{\vdim}$.}
in addition, $\lyap$ is continuously differentiable on $\nhd_{\saddles} \setminus \mfld$.

\item
For all $\point\in\nhd_{\saddles}$, we have
\begin{equation}
\label{eq:energy-rder}
\nabla^{+}\lyap(\point) [\nabla\obj(\point)]
	\leq - \beta \lyap(\point).
\end{equation}
In particular, for all $\point\in\nhd_{\saddles}\setminus\mfld$, we have:
\begin{equation}
\label{eq:energy-grad}
\braket{\nabla\lyap(\point)}{\nabla\obj(\point)}
	\leq -\beta \lyap(\point)
\end{equation}

\item
There exists a constant $\alpha>0$ such that, for all $\point\in\nhd_{\saddles}$ and all sufficiently small $\vvec\in\R^{\vdim}$, we have
\begin{equation}
\label{eq:energy-discrete}
\lyap(\point + \vvec)
	\geq \lyap(\point)
		+ \nabla^{+}\lyap(\point)[\vvec]
		- \frac{\alpha}{2} \norm{\vvec}^{2}.
\end{equation}

\item
There exists a constant $\beta>0$ such that, for all $\vvec\in\R^{\vdim}$, we have:
\begin{subequations}
\begin{alignat}{2}
\norm{\nabla\lyap(\point)}
	&\geq \beta
		&\qquad
		&\text{for all $\point\in\nhd_{\saddles}\setminus\mfld$},
\shortintertext{and}
\nabla^{+}\lyap(\point)[\vvec]
	&\geq \beta \norm{\prox{\point}{\vvec} - \vvec}
		&\qquad
		&\text{for all $\point\in\nhd_{\saddles}\cap\mfld$}.
\end{alignat}
\end{subequations}
\end{enumerate}
\end{proposition}

\Cref{prop:energy} follows from \citet[Proposition 9.5]{Ben99}, so we do not present a proof.
More important for our purposes are the following immediate consequences thereof:
\begin{enumerate}

\item
By \eqref{eq:energy-grad}, the energy $\lyap(\orbit{\ctime})$ of a solution orbit $\orbit{\ctime}$ of \eqref{eq:GD} will grow at a (locally) geometric rate if $\orbit{\ctime}$ doesn't already lie in the center stable manifold $\mfld$ of $\saddles$.
This means that \aclp{APT} of \eqref{eq:GD} that do not lie on $\mfld$ for arbitrarily long windows of time will also escape $\mfld$ (and hence $\saddles$).

\item
The bound \eqref{eq:energy-discrete} provides the basis for a discrete-time version of the above argument:
as long as $\curr$ is sufficiently close to $\saddles$, the energy before and after a stochastic gradient step will be linked as
\begin{equation}
\label{eq:energy-SGD}
\lyap(\next)
	\geq \lyap(\curr)
		+ \beta \curr[\step] \lyap(\curr)
		- \curr[\step] \curr[\sbias]
		- \frac{\alpha\curr[\step]^{2}}{2} \norm{\curr[\signal]}^{2},
\end{equation}
where $\curr[\sbias]$ is an additive noise term which is non-antagonistic in expectation.
This means that, on average, the iterates $\curr[\lyap] \equiv \lyap(\curr)$ will grow at a (locally) geometric rate, so $\curr$ cannot remain in the vicinity of $\saddles$ for very long periods.
\end{enumerate}

To make the above precise, we will need to invoke the probabilistic estimates stated in \cref{sec:Pemantle}.
We do so in the following section.

\subsection{Avoidance of saddle-point manifolds}
\label{sec:Benaim}

For convenience, we begin by restating our main avoidance result below:

\avoidance*

\begin{proof}
Our proof follows the arguments of \citet{BH95}, suitably adapted to our setting.
To begin, let $\nhd_{\saddles}$ be the compact neighborhood of $\saddles$ identified in \cref{prop:energy}
and
assume without loss of generality that $\init\in\nhd_{\saddles}$.
We may then define the exit time from $\nhd_{\saddles}$ as
\begin{equation}
\label{eq:exit}
T_{\saddles}
	= \inf\setdef{\run\geq\start}{\run\notin\nhd_{\saddles}}.
\end{equation}
We will prove our claim by showing that $T_{\saddles} < \infty$ with probability $1$.

To that end, consider the process
\begin{equation}
\label{eq:increment}
\next[Y]
	= \begin{cases}
		\lyap(\next) - \lyap(\curr)
			&\quad
			\text{if $\run\leq T_{\saddles}$},
		\\
		\curr[\step]
			&\quad
			\text{otherwise},
		\end{cases}
\end{equation}
with $\lyap(\state_{0}) \equiv 0$ by convention.
Heuristically, $\curr[Y]$ measures the change in energy of $\curr$ as long as it remains in $\nhd_{\saddles}$;
subsequently, for book-keeping purposes, it is incremented by a token amount of $\curr[\step]$ per iteration once $\curr$ exits $\nhd_{\saddles}$.
To make this idea more formal, let
\begin{equation}
\curr[\lyap]
	= \sum_{\runalt=\start}^{\run} \iter[Y]
\end{equation}
so
$\curr[\lyap] = \lyap(\curr)$ if $\curr\in\nhd_{\saddles}$
while
$\curr[\lyap] = \Theta(\curr[\proper])$ after $\curr$ exits $\nhd_{\saddles}$.

Assume now that $\curr\in\nhd_{\saddles}$ for all $\run$ (\ie $T_{\saddles}=\infty$).
By \cref{thm:convergence}, every limit point $\olim$ of $\curr$ must be contained in $\saddles$, so, by \eqref{eq:pot-posdef}, we must have $\lim_{\run\to\infty} \curr[\lyap] = 0$.
Hence, to establish our claim, it suffices to show that $\probof{\curr[\lyap]\to0} = 0$.
We will do this by showing that $\curr[Y]$ defined as in \eqref{eq:increment} satisfies the requirements of \cref{lem:Pem-large,lem:Pem-small}.

We begin with the conditions required by \cref{lem:Pem-small} for the case $1/2<\pexp\leq1$:
\begin{enumerate}
\item
For the condition $\abs{\curr[Y]} = \bigoh(1/\run^{\pexp})$ of \cref{lem:Pem-small}, the claim is tautological if $\run > T_{\saddles}$.
Otherwise, if $\run \leq T_{\saddles}$, note that
\begin{equation}
\norm{\next - \curr}
	= \curr[\step] \norm{\curr[\signal]}
	\leq \curr[\step] \bracks{\norm{\nabla\obj(\curr)} + \norm{\curr[\noise]}}
	\leq \curr[\step] (\gbound + \noisepar)
\end{equation}
by \cref{asm:reg,asm:noise} (recall here that we are taking $\qexp=\infty$ in \cref{asm:noise}).
Since $\curr\in\nhd_{\saddles}$ as long as $\run\leq T_{\saddles}$, and given that $\lyap$ is continuously differentiable on $\nhd_{\saddles}$ (and hence Lipschitz continuous therein), we also have:
\begin{equation}
\abs{\lyap(\next) - \lyap(\curr)}
	= \bigoh(\norm{\next - \curr})
	= \bigoh(\curr[\step])
	= \bigoh(1/\run^{\pexp}),
\end{equation}
as claimed.

\item
For the condition $\oneof{\curr[\lyap] > b/\run^{\pexp}} \exof{\next[Y] \given \curr[\filter]} \geq 0$, note first that if $\run > T_{\saddles}$, then  $\curr[Y] = \curr[\step]$, so
\begin{equation}
\label{eq:Pemantle-1}
\oneof{\run > T_{\saddles}} \exof{\next[Y] \given \curr[\filter]}
	\geq \oneof{\run > T_{\saddles}} \curr[\step]
	> 0.
\end{equation}
Otherwise, if $\run\leq T_{\saddles}$, we have $\curr\in\nhd_{\saddles}$, so \cref{prop:energy} yields
\begin{equation}
\label{eq:lyap-bound1}
\next[Y]
	= \lyap(\next) - \lyap(\curr)
	\geq \beta \curr[\step] \lyap(\curr)
		- \curr[\step] \curr[\sbias]
		- 2 \alpha \curr[\step]^{2} (\gbound^{2} + \noisepar^{2}),
\end{equation}
where we set
\begin{equation}
\label{eq:bias}
\curr[\sbias]
	= \nabla^{+}\lyap(\curr) [\curr[\noise]]
\end{equation}
and
 used the estimate $\norm{\curr[\signal]}^{2} = \norm{\nabla\obj(\curr) + \curr[\noise]}^{2} \leq \bracks{\norm{\nabla\obj(\curr)}^{2} + \norm{\curr[\noise]}^{2}} \leq 2(\gbound^{2} + \noisepar^{2})$ (compare also with \eqref{eq:energy-SGD} and the surrounding discussion).
By the conditional Jensen inequality and the definition of $\nabla^{+}\lyap(\point)$, we have
\begin{equation}
\exof{\curr[\sbias] \given \curr[\filter]}
	= \exof{\nabla^{+}\lyap(\curr) [\curr[\noise]] \given \curr[\filter]}
		\geq \nabla^{+}\lyap(\curr) [\exof{\curr[\noise] \given \curr[\filter]}]
		= 0.
\end{equation}
which, in turn, implies that
\begin{equation}
\label{eq:lyap-bound2}
\oneof{\run\leq T_{\saddles}}
	\exof{\next[Y] \given \curr[\filter]}
	\geq \curr[\step]
		\oneof{\run\leq T_{\saddles}}
		\bracks[\big]{\beta\lyap(\curr) - 2\alpha(\gbound^{2} + \noisepar^{2}) \curr[\step]}.
\end{equation}
Hence, taking $b>0$ such that $b\beta/\run^{\pexp} = 2\alpha(\gbound^{2} + \noisepar^{2}) \curr[\step]$, and recalling that $\curr[\lyap] = \lyap(\curr)$ if $\run \leq T_{\saddles}$, we get
\begin{align}
\label{eq:Pemantle-2}
\oneof{\curr[\lyap] > b/\run^{\pexp}}
	\oneof{\run\leq T_{\saddles}}
	\exof{\next[Y] \given \curr[\filter]}
	&\geq \oneof{\curr[\lyap] > b/\run^{\pexp} \wedge \run\leq T_{\saddles}}
		\bracks[\big]{\beta\curr[\lyap] - 2\alpha(\gbound^{2} + \noisepar^{2}) \curr[\step]}
	\notag\\
	&\geq \oneof{\curr[\lyap] > b/\run^{\pexp} \wedge \run\leq T_{\saddles}}
		\bracks[\big]{b\beta/\run^{\pexp} - 2\alpha(\gbound^{2} + \noisepar^{2}) \curr[\step]}
	\notag\\
	&\geq 0. 
\end{align}
Thus, combining the above, we conclude that the specific conditions required to apply \cref{lem:Pem-small} are satisfied.
\end{enumerate}

We are left to establish the general condition \eqref{eq:energy-unst} which is required to apply both \cref{lem:Pem-small,lem:Pem-large};
the proof is the same for all $\pexp\in(0,1]$, so we no longer assume $1/2 < \pexp \leq 1$ below.
To begin, note that
\begin{align}
\exof{\next[\lyap]^{2} - \curr[\lyap]^{2} \given \curr[\filter]}
	&= \exof{\next[Y]^{2} \given \curr[\filter]}
		+ 2 \curr[\lyap] \exof{\next[Y] \given \curr[\filter]}
	\notag\\
	&= \exof{\next[Y]^{2} \given \curr[\filter]}
		+ 2 \curr[\lyap] \oneof{\curr[\lyap] \leq b/\run^{\pexp}} \exof{\next[Y] \given \curr[\filter]}
	\notag\\
	&\hspace{3em}
		+ 2 \curr[\lyap] \oneof{\curr[\lyap] > b/\run^{\pexp}} \exof{\next[Y] \given \curr[\filter]}
	\notag\\
	&\geq \exof{\next[Y]^{2} \given \curr[\filter]}
		+ 2 \curr[\lyap] \oneof{\curr[\lyap] \leq b/\run^{\pexp}} \exof{\next[Y] \given \curr[\filter]},
\end{align}
where, in the last line, we used the inequalities proved in the previous paragraph, namely \eqref{eq:Pemantle-1} and \eqref{eq:Pemantle-2}.
To proceed, recall that $\curr[Y] = \curr[\step] > 0$ if $\run > T_{\saddles}$, so, by \eqref{eq:lyap-bound2} we get
\begin{align}
\exof{\next[\lyap]^{2} - \curr[\lyap]^{2} \given \curr[\filter]}
	&\geq \exof{\next[Y]^{2} \given \curr[\filter]}
	\notag\\
	&\hspace{3em}
		+ 2 \curr[\lyap]
			\oneof{\curr[\lyap] \leq b/\run^{\pexp}}
			\oneof{\run \leq T_{\saddles}}
			\exof{\next[Y] \given \curr[\filter]}
	\notag\\
	&\geq \exof{\next[Y]^{2} \given \curr[\filter]}
	\notag\\
	&\hspace{3em}
		+ 2 \curr[\step] \curr[\lyap]
			\oneof{\curr[\lyap] \leq b/\run^{\pexp}}
			\oneof{\run \leq T_{\saddles}}
			\bracks[\big]{\beta\curr[\lyap] - 2\alpha(\gbound^{2} + \noisepar^{2}) \curr[\step]}
	\notag\\
	&\geq \exof{\next[Y]^{2} \given \curr[\filter]}
		- 2 \curr[\step] \cdot b/\run^{\pexp} \cdot 2\alpha(\gbound^{2} + \noisepar^{2}) \curr[\step]
	\notag\\
	&= \exof{\next[Y]^{2} \given \curr[\filter]}
		- 4 \alpha^{2} \beta^{-1} (\gbound^{2} + \noisepar^{2})^{2} \curr[\step]^{3}.
\end{align}

In view of the above, to establish \eqref{eq:energy-unst}, it suffices to show that $\exof{\next[Y]^{2} \given \curr[\filter]} \geq B\curr[\step]^{2}$ for some $B>0$ and sufficiently large $\run$.
In this regard, Jensen's inequality gives
\begin{equation}
\exof{\next[Y]^{2} \given \curr[\filter]}
	\geq \exof{\next[Y]^{+} \given \curr[\filter]}^{2}
\end{equation}
so it suffices to show that $\exof{\next[Y]^{+} \given \curr[\filter]} = \Omega(\curr[\step])$.
This is trivial if $\run > T_{\saddles}$, so we are left to treat the case $\run \leq T_{\saddles}$.
For this case, \eqref{eq:lyap-bound1} gives
\begin{equation}
\label{eq:pospart-bound}
\oneof{\run \leq T_{\saddles}} \exof{\next[Y]^{+} \given \curr[\filter]}
	\geq \oneof{\run \leq T_{\saddles}} \curr[\step] \exof{\curr[\sbias]^{-} \given \curr[\filter]}
		- 2 \oneof{\run \leq 2T_{\saddles}} \alpha (\gbound^{2} + \noisepar^{2}) \curr[\step]^{2}
\end{equation}
meaning that we need to focus on the expectation $\exof{\curr[\sbias]^{-} \given \curr[\filter]}$.

We consider two further cases (this is where \cref{asm:noise} kicks in and plays a crucial role).
First, if $\curr\notin\mfld$, \cref{prop:energy} and \cref{asm:exciting} applied to $\vvec = -\nabla\lyap(\curr) / \norm{\nabla\lyap(\curr)}$ give
\begin{align}
\oneof{\run \leq T_{\saddles} \wedge \curr\notin\mfld}
	\exof{\curr[\sbias]^{-} \given \curr[\filter]}
	&= \oneof{\run \leq T_{\saddles} \wedge \curr\notin\mfld}
		\exof{ \braket{-\nabla\lyap(\curr)}{\curr[\noise]}^{+} \given \curr[\filter]}
	\notag\\
	&\geq \oneof{\run \leq T_{\saddles} \wedge \curr\notin\mfld}
		\cdot \const \norm{\nabla\lyap(\curr)}
	\notag\\
	&\geq \beta \const \oneof{\run \leq T_{\saddles} \wedge \curr\notin\mfld}.
\end{align}
Otherwise, if $\curr\in\mfld$ (which, heuristically, should only happen with probability $0$), choose a unit normal vector $\curr[\unitvec]$ such that
\begin{equation}
\braket{\curr[\unitvec]}{\tvec}
	= 0
	\quad
	\text{for all $\tvec\in\tangent{\curr}{\mfld}$}.
\end{equation}
Since the projector $\proxmap_{\curr}$ defined in \eqref{eq:proj-fiber} takes values in $\tangent{\curr}{\mfld}$, we will have $\braket{\curr[\unitvec]}{\prox{\curr}{\curr[\noise]}} = 0$, and hence:
\begin{equation}
\label{eq:normal}
\braket{\curr[\unitvec]}{\curr[\noise]}
	= \braket{\curr[\unitvec]}{\curr[\noise] - \prox{\curr}{\curr[\noise]}}.
\end{equation}
Therefore, by \cref{prop:energy}, we get the chain of inequalities:
\usetagform{comment}
\begin{align}
\exof{ \bracks{\nabla^{+}\lyap(\curr)[\curr[\noise]]}^{-} \given \curr[\filter]}
	&\geq \beta \exof{ \norm{\prox{\curr}{\curr[\noise]} - \curr[\noise]} \given \curr[\filter]}
	\tag{by \cref{prop:energy}}
	\\
	&\geq \beta \exof{ \braket{\curr[\unitvec]}{\curr[\noise] - \prox{\curr}{\curr[\noise]}}^{+} \given \curr[\filter]}
	\tag{by Cauchy\textendash Schwarz}
	\\
	&= \beta \exof{ \braket{\curr[\unitvec]}{\curr[\noise]}^{+} \given \curr[\filter]}
	\tag{by (\ref{eq:normal})}
	\\
	&\geq \beta \const
	\tag{by \cref{asm:exciting}}
\end{align}
\usetagform{default}%
valid on the event $\{\run \leq T_{\saddles} \wedge \curr\in\mfld\}$.

Putting together all of the above, we finally get
\begin{equation}
\oneof{\run \leq T_{\saddles}}
	\exof{\curr[\sbias]^{-} \given \curr[\filter]}
	\geq \oneof{\run \leq T_{\saddles}} \beta\const
\end{equation}
and hence, by \eqref{eq:pospart-bound}:
\begin{equation}
\exof{\next[Y]^{+} \given \curr[\filter]}
	\geq \beta\const \curr[\step]
		- 2\alpha(\gbound^{2} + \noisepar^{2}) \curr[\step]^{2}
	= \Omega(\curr[\step])
\end{equation}
on the event $\{\run \leq T_{\saddles}\}$.
This completes our proof.
\end{proof}

\section{Rates of convergence}
\label{app:rates}

Our aim in this appendix is to establish the rate of convergence of \eqref{eq:SGD} to local minima that are regular in the sense of Hurwicz, \ie $\hmat(\sol) \mg 0$.
For convenience, we restate the relevant result below:

\rate*

\para{Auxiliary results}
The proof of \cref{thm:rate} requires several ancillary results, which we state and prove below.
The first is a lemma on numerical sequences, usually attributed to \citet{Chu54}:

\begin{lemma}[\citealp{Chu54}, Lemma 1]
\label{lem:Chung}
Let $\curr[a]$, $\run=\running$, be a non-negative sequence such that
\begin{equation}
\label{eq:Chung-cond}
\next[a]
	\leq \bracks*{1 - \frac{\pcoef}{(\run+\offset)^{\pexp}}} \curr[a]
		+ \frac{\rcoef}{(\run+\offset)^{\pexp+\rexp}}
\end{equation}
where $\pexp\in(0,1]$, $\rexp>0$ and $P,R>0$.
Then:
\begin{subequations}
\label{eq:Chung-bound}
\begin{enumerate}
\item
If $\pexp<1$, we have
\begin{equation}
\label{eq:Chung-slow}
\curr[a]
	\leq \frac{\rcoef}{\pcoef} \frac{1}{\run^{\rexp}}
	+ o\parens*{\frac{1}{\run^{\rexp}}}.
\end{equation}
\item
If instead $\pexp=1$ and $\pcoef > \rexp$, we have
\begin{equation}
\label{eq:Chung-fast}
\curr[a]
	\leq \frac{\rcoef}{\pcoef - \rexp} \frac{1}{\run}
	+ o\parens*{\frac{1}{\run}}.
\end{equation}
\end{enumerate}
\end{subequations}
\end{lemma}

The next ingredient of the proof of \cref{thm:rate} provides a handle on the local behavior of $\obj$ near a regular minimizer:

\begin{lemma}
\label{lem:Minty}
Let $\sol$ be a regular minimum of $\obj$.
Then,
there exists a convex compact neighborhood $\cpt$ of $\sol$ and constants $\qmin,\qmaj > 0$ \textpar{possibly depending on $\cpt$} such that
\begin{equation}
\label{eq:qbounds}
\qmin \norm{\point - \sol}^{2}
	\leq \braket{\nabla\obj(\point)}{\point - \sol}
	\leq \qmaj \norm{\point - \sol}^{2}
	\quad
	\text{for all $\point\in\cpt$}.
\end{equation}
\end{lemma}

\begin{proof}
Let $\cpt$ be a sufficiently small convex compact neighborhood of $\sol$ such that $\hmat(\point) \mg 0$ for all $\point\in\cpt$ (that such a neighborhood exists is a consequence of the regularity of $\sol$ and the smoothness of $\obj$).
Then, by compactness, there exist constants $\qmin,\qmaj$ such that $\qmin \eye \mleq \hmat(\point) \mleq \qmaj\eye$ for all $\point\in\cpt$.
Moreover, for all $\point\in\cpt$, we have
\begin{equation}
\nabla\obj(\point)
	= (\point - \sol)^{\top} \int_{0}^{1} \hmat(\sol + \ctime (\point - \sol)) \dd\ctime,
\end{equation}
where we used the fact that $\nabla\obj(\sol) = 0$ (since $\sol$ is a minimizer of $\obj$).
Hence, multiplying both sides by $\point - \sol$, the mean value theorem for integrals yields:
\begin{align}
\braket{\nabla\obj(\point)}{\point - \sol}
	&= \int_{0}^{1}
		(\point - \sol)^{\top}
		\hmat(\sol + \ctime (\point - \sol))
		(\point - \sol)
	 \dd\ctime
	 \notag\\
	 &= (\point - \sol)^{\top}
		\hmat(\pointalt)
		(\point - \sol)
\end{align}
for some $\pointalt \in [\sol,\point]$.
Since $\qmin \eye \mleq \hmat(\pointalt) \mleq \qmaj\eye$, our claim follows.
\end{proof}

Thanks to \cref{lem:Minty}, we obtain the following recursive estimate for \eqref{eq:SGD}:

\begin{proposition}
\label{prop:dist-loc}
Let $\sol$ be a regular minimum of $\obj$ and let $\cpt$ and $\qmin$ be as in \cref{lem:Minty}.
Assume moreover that $\curr\in\cpt$ for some $\run \geq 1$ and let
\begin{equation}
\label{eq:dist-curr}
\curr[\breg]
	= \frac{1}{2} \norm{\curr - \sol}^{2}.
\end{equation}
We then have:
\begin{equation}
\label{eq:dist-loc}
\next[\breg]
	\leq (1 - 2\qmin\curr[\step]) \curr[\breg]
	+ \curr[\step] \curr[\snoise]
	+ \tfrac{1}{2} \curr[\step]^{2} \norm{\curr[\signal]}^{2},
\end{equation}
where $\curr[\snoise] = -\braket{\curr[\noise]}{\curr - \sol}$ is a \acl{MDS}.
\end{proposition}

\begin{proof}
Recall first that $\next = \curr - \curr[\step] (\nabla\obj(\curr) + \curr[\noise])$ where $\curr[\noise]$ is the gradient error at $\curr$.
Then, by the definition of $\curr[\breg]$, we have:
\begin{align}
\next[\breg]
	= \tfrac{1}{2} \norm{\next - \sol}^{2}
	&= \tfrac{1}{2} \norm{\curr - \sol - \curr[\step] \curr[\signal]}^{2}
	\notag\\
	&= \tfrac{1}{2} \norm{\curr - \sol}^{2}
		- \curr[\step] \braket{\curr[\signal]}{\curr - \sol}
		+ \tfrac{1}{2} \curr[\step]^{2} \norm{\curr[\signal]}^{2}
	\notag\\
	&= \curr[\breg]
		- \curr[\step] \braket{\nabla\obj(\curr)}{\curr - \sol}
		- \curr[\step] \braket{\curr[\noise]}{\curr - \sol}
		+ \tfrac{1}{2}\curr[\step]^{2} \norm{\curr[\signal]}^{2}
	\notag\\
	&\leq \curr[\breg]
		- \qmin \curr[\step] \norm{\curr - \sol}^{2}
		+ \curr[\step] \curr[\snoise]
		+ \tfrac{1}{2}\curr[\step]^{2} \norm{\curr[\signal]}^{2}
	\notag\\
	&= (1 - 2\qmin\curr[\step]) \curr[\breg]
		+ \curr[\step] \curr[\snoise]
		+ \tfrac{1}{2}\curr[\step]^{2} \norm{\curr[\signal]}^{2}
\end{align}
where the second-to-last line follows from \cref{lem:Minty}.
Since $\exof{\curr[\snoise] \given \curr[\filter]} = \braket{\exof{\curr[\noise] \given \curr[\filter]}}{\curr - \sol} = 0$, our claim follows (recall here that, by definition, $\curr[\noise]$ is not $\curr[\filter]$-measurable but $\curr$ is).
\end{proof}

With these basic results at our disposal, the proof of \cref{thm:rate} will roughly follow the technical trajectory outlined below:
\begin{enumerate}
\item
By \cref{prop:dist-loc}, $\curr[\breg]$ grows at most by $\curr[\step] \curr[\snoise] + \tfrac{1}{2} \curr[\step]^{2} \norm{\curr[\signal]}^{2}$ at each step.
This quantity can be big for any given $\run$ but we will show that, with high probability (and, in particular, with probability at least $1-\delta$), the aggregation of these errors remains controllably small.
This will be the most technical and involved part of our argument.

\item
Using the above, we will show that, with probability at least $1-\delta$, $\curr[\breg]$ cannot grow more than a token quantity $\eps$.
As a result, if the initial distance to $\sol$ is not too big, $\curr$ will remain in a neighborhood thereof for all time.

\item
For the final part of the theorem, we will condition on this event to map \eqref{eq:dist-loc} to a recursion of the form \eqref{eq:Chung-cond}, and we will subsequently employ \cref{lem:Chung} to obtain the stated result.
The main problem here is that, after conditioning, the noise in \eqref{eq:dist-loc} is no longer zero-mean, so we will need to adapt our analysis to the new noise distribution.
\end{enumerate}
We make all this precise below.
For convenience, we focus on the case $\pexp>1/2$;
the case $\pexp\in(2/(\qexp+2),1/2]$ follows by modifying the arguments that follow with the Hölder estimates we introduced in the proof of \cref{lem:subsequence}.

\para{Controlling the error terms}

We begin by encoding the error terms in \eqref{eq:dist-loc} as
\begin{align}
\label{eq:err-mart}
\curr[M]
	&= \sum_{\runalt=\start}^{\run} \iter[\step] \iter[\snoise]
\shortintertext{and}
\label{eq:err-subm}
\curr[S]
	&= \frac{1}{2} \sum_{\runalt=\start}^{\run} \iter[\step]^{2} \norm{\iter[\signal]}^{2}
\end{align}
Since $\exof{\curr[\snoise] \given \curr[\filter]} = 0$, we have $\exof{\curr[M] \given \curr[\filter]} = \prev[M]$, so $\curr[M]$ is a zero-mean martingale;
likewise, $\exof{\curr[S] \given \curr[\filter]} \geq \prev[S]$, so $\curr[S]$ is a submartingale.
Interestingly, even though $\curr[M]$ is more ``neutral'' as an error (because $\curr[\snoise]$ is zero-mean), it is more difficult to control because the variance of its increments is
\begin{equation}
\exof{\norm{\curr[\step] \curr[\snoise]}^{2} \given \curr[\filter]}
	= \curr[\step]^{2} \exof{\braket{\curr[\noise]}{\curr - \sol}^{2} \given \curr[\filter]}
\end{equation}
and this last quantity can become arbitrarily big if $\curr$ does not remain in the vicinity of $\sol$ (which is what we are trying to prove).
Because of this, we need to take a less direct, step-by-step approach to bound the total error increments \emph{conditioned} on the event that $\curr$ remains close to $\sol$.
Our approach builds on a range of ideas and techniques due to \citet{HIMM19,HIMM20} and \citet{MZ19}.

We begin by introducing the ``cumulative mean square'' error
\begin{equation}
\label{eq:err-tot}
\curr[R]
	= \curr[M]^{2} + \curr[S].
\end{equation}
By construction, we have
\begin{align}
\label{eq:err-tot-upd}
\curr[R]
	&= (\prev[M] + \curr[\step]\curr[\snoise])^{2}
		+ \prev[S] + \tfrac{1}{2} \curr[\step]^{2} \norm{\curr[\signal]}^{2}
	\notag\\
	&= \prev[R]
		+ 2 \prev[M] \curr[\step] \curr[\snoise]
		+ \curr[\step]^{2} \curr[\snoise]^{2}
		+ \tfrac{1}{2} \curr[\step]^{2} \norm{\curr[\signal]}^{2}
\end{align}
and hence, after taking expectations:
\begin{align}
\exof{\curr[R] \given \curr[\filter]}
	&= \prev[R]
		+ 2 \prev[M] \curr[\step]  \exof{\curr[\snoise] \given \curr[\filter]}
		+ \curr[\step]^{2} \exof{
			\curr[\snoise]^{2}
			+ \tfrac{1}{2} \norm{\curr[\signal]}^{2}
			\given \curr[\filter] }
	\geq \prev[R]
\end{align}
\ie $\curr[R]$ is a submartingale.
To condition it further, let $\nhd$ be a neighborhood of $\sol$, let $\eps>0$,
and define the events
\begin{align}
\label{eq:evt-stay}
\curr[\event]
	\equiv \curr[\event](\nhd)
	&= \braces*{\curr\in\nhd \; \text{for all $\runalt=\running,\run$}}
\shortintertext{and}
\label{eq:evt-small}
\curr[\eventalt]
	\equiv \curr[\eventalt](\eps)
	&= \braces*{\iter[R] \leq \eps \; \text{for all $\runalt = \running,\run$}}.
\end{align}
By definition, we also have $\event_{0} = \eventalt_{0} = \samples$ (because the set-building index set for $\runalt$ is empty in this case, and every statement is true for the elements of the empty set).
These events will play a crucial role in the sequel as indicators of whether $\curr$ has escaped the vicinity of $\sol$ or not.

To proceed, we will instantiate $\nhd$ and $\eps$ in the definition of $\event$ and $\eventalt$ respectively as follows.
First, for \eqref{eq:evt-stay}, we will choose a neighborhood $\nhd$ contained in the convex compact neighborhood $\cpt$ of $\sol$ (whose existence is guaranteed by \cref{lem:Minty});
in particular, this implies that \eqref{eq:qbounds} holds for all $\point\in\nhd$.
Moreover, with a fair degree of hindsight, we will also choose $\eps>0$ such that
\begin{equation}
\label{eq:eps-bound}
	\setdef{\point\in\points}{\norm{\point - \sol}^{2} \leq 4\eps + 2\sqrt{\eps}}
	\subseteq \nhd.
\end{equation}
and we will assume that $\init$ is initialized in a neighborhood $\nhd_{1} \subseteq \nhd$ such that
\begin{equation}
\label{eq:nhd-init}
\nhd_{1}
	\subseteq \setdef{\point\in\points}{\norm{\point - \sol}^{2} \leq 2 \eps}
\end{equation}
These will be the neighborhoods $\nhd$ and $\nhd_{1}$ whose existence is postulated by \cref{thm:rate}.
Then, with all this in hand, we have:

\begin{lemma}
\label{lem:events}
Let $\sol$ be a regular minimizer of $\obj$ as above and assume that \cref{asm:noise} holds.
Then, for all $\run=\running$, we have:
\begin{enumerate}
\item
$\next[\event] \subseteq \curr[\event]$
and
$\next[\eventalt] \subseteq \curr[\eventalt]$.
\item
$\prev[\eventalt] \subseteq \curr[\event]$.
\item
Consider the ``large noise'' event
\begin{align}
\label{eq:evt-bad}
\curr[\tilde\eventalt]
	&\equiv \prev[\eventalt] \setminus \curr[\eventalt]
	= \prev[\eventalt] \cap \{\curr[R] > \eps\}
	\notag\\
	&= \braces*{\text{$\iter[R] \leq \eps$ for all $\runalt=\running,\run-1$ and $\curr[R] > \eps$}},
\end{align}
and
let $\curr[\tilde R] = \curr[R] \one_{\prev[\eventalt]}$ denote the cumulative error subject to the noise being ``small'' until time $\run$.
Then:
\begin{equation}
\label{eq:err-tot-cond}
\exof{\curr[\tilde R]}
	\leq \exof{\prev[\tilde R]}
		+ \bracks{\gbound^{2} + (1+\radius_{\nhd}^{2})\noisevar} \curr[\step]^{2}
		- \eps \probof{\prev[\tilde\eventalt]},
\end{equation}
where $\radius_{\nhd} = \sup_{\point\in\nhd} \norm{\point - \sol}$
and, by convention,
we write $\tilde\eventalt_{0} = \varnothing$ and $\tilde R_{0} = 0$.
\end{enumerate}
\end{lemma}

\begin{remark*}
In the above (and what follows), the notation $\one_{A}$ is used to indicate the logical indicator of an event $A\subseteq\samples$, \ie $\one_{A}(\sample) = 1$ if $\sample\in A$ and $\one_{A}(\sample) = 0$ otherwise.
\end{remark*}

\begin{proof}
The first claim is obvious.
For the second, we proceed inductively:
\begin{enumerate}
\item
For the base case $\run=\start$, we have $\init[\event] = \{\init \in \nhd \} \supseteq \{\init \in \nhd_{1} \} = \samples$ because $\init$ is initialized in $\nhd_{1} \subseteq \nhd$.
Since $\eventalt_{0} = \samples$, our claim follows.

\item
For the inductive step, assume that $\prev[\eventalt] \subseteq \curr[\event]$ for some $\run\geq\start$.
To show that $\curr[\eventalt] \subseteq \next[\event]$, fix a realization in $\curr[\eventalt]$ so $\iter[R] \leq \eps$ for all $\runalt=\running,\run$.
Since $\curr[\eventalt] \subseteq \prev[\eventalt]$, the inductive hypothesis posits that $\curr[\event]$ also occurs, \ie $\iter\in\nhd$ for all $\runalt=\running,\run$;
hence, it suffices to show that $\next\in\nhd$.


To that end, given that $\iter\in\nhd\subseteq\cpt$ for all $\runalt=\running\run$, the distance estimate \eqref{eq:dist-loc} readily gives
\begin{equation}
\breg_{\runalt+1}
	\leq \iter[\breg]
	+ \iter[\step] \iter[\snoise]
	+ \tfrac{1}{2} \iter[\step]^{2} \norm{\iter[\signal]}^{2}
	\quad
	\text{for all $\runalt = \running\run$}.
\end{equation}
Therefore, after telescoping, we obtain
\begin{equation}
\next[\breg]
	\leq \init[\breg]
		+ \curr[M]
		+ \curr[S]
	\leq \init[\breg]
		+ \sqrt{\curr[R]}
		+ \curr[R]
	\leq \eps
		+ \sqrt{\eps}
		+ \eps
	= 2\eps + \sqrt{\eps}
\end{equation}
by the inductive hypothesis.
We conclude that $\norm{\next - \sol}^{2} = 2\next[\breg] \leq 4\eps + 2\sqrt{\eps}$, so $\next\in\nhd$ and the induction is complete.
\end{enumerate}

For our third claim, we decompose $\curr[\tilde R]$ as
\begin{align}
\label{eq:err-tot-cond1}
\curr[\tilde R]
	= \curr[R] \one_{\prev[\eventalt]}
	&= \prev[R] \one_{\prev[\eventalt]}
		+ (\curr[R] - \prev[R]) \one_{\prev[\eventalt]}
	\notag\\
	&= \prev[R] \one_{\preprev[\eventalt]}
		- \prev[R] \one_{\prev[\tilde\eventalt]}
		+ (\curr[R] - \prev[R]) \one_{\prev[\eventalt]},
	\notag\\
	&= \prev[\tilde R]
		+ (\curr[R] - \prev[R]) \one_{\prev[\eventalt]}
		- \prev[R] \one_{\prev[\tilde\eventalt]},
\end{align}
where we used the fact that $\prev[\eventalt] = \preprev[\eventalt] \setminus \prev[\tilde\eventalt]$ so $\one_{\prev[\eventalt]} = \one_{\preprev[\eventalt]} - \one_{\prev[\tilde\eventalt]}$ (recall here that $\prev[\eventalt] \subseteq \preprev[\eventalt]$).
Now, to proceed, \eqref{eq:err-tot-upd} yields
\begin{align}
\curr[R] - \prev[R]
		= 2 \prev[M] \curr[\step] \curr[\snoise]
		+ \curr[\step]^{2} \curr[\snoise]^{2}
		+ \tfrac{1}{2} \curr[\step]^{2} \norm{\curr[\signal]}^{2}
\end{align}
so
\begin{subequations}
\begin{align}
\exof{(\curr[R] - \prev[R]) \one_{\prev[\eventalt]}}
	&\label{eq:err-tot-zero}
		= 2 \curr[\step] \exof{\prev[M]\curr[\snoise] \one_{\prev[\eventalt]}}
	\\
	&\label{eq:err-tot-noise}
		+ \curr[\step]^{2} \exof{\curr[\snoise]^{2} \one_{\prev[\eventalt]}}
	\\
	&\label{eq:err-tot-signal}
		+ \tfrac{1}{2} \curr[\step]^{2} \exof{\norm{\curr[\signal]}^{2} \one_{\prev[\eventalt]}}
\end{align}
\end{subequations}
However, since $\prev[\eventalt]$ and $\prev[M]$ are both $\curr[\filter]$-measurable, we have the following estimates:
\begin{enumerate}
\item
For the noise term in \eqref{eq:err-tot-zero}, the second part of \cref{prop:dist-loc} gives:
\begin{equation}
\exof{\prev[M] \curr[\snoise] \one_{\prev[\eventalt]}}
	= \exof{\prev[M] \one_{\prev[\eventalt]} \exof{\curr[\snoise] \given \curr[\filter]}}
	= 0.
\end{equation}

\item
The term \eqref{eq:err-tot-noise} is where the conditioning on $\prev[\eventalt]$ plays the most important role because it allows us to control the distance $\norm{\curr - \sol}$.
Specifically, we have:
\usetagform{comment}
\begin{align}
\exof{\curr[\snoise]^{2} \one_{\prev[\eventalt]}}
	&= \exof{\one_{\prev[\eventalt]} \exof{ \braket{\curr[\noise]}{\curr - \sol}^{2} \given \curr[\filter]} }
	\notag\\
	&\leq \exof{\one_{\prev[\eventalt]} \norm{\curr-\sol}^{2} \exof{\norm{\curr[\noise]}^{2} \given \curr[\filter]}}
	\tag{by Cauchy\textendash Schwarz}
	\\
	&\leq \exof{\one_{\curr[\event]} \norm{\curr-\sol}^{2} \exof{\norm{\curr[\noise]}^{2} \given \curr[\filter]}}
	\tag{because $\prev[\eventalt] \subseteq \curr[\event]$}
	\\
	&\leq \radius_{\nhd}^{2} \noisevar.
	\tag{by \cref{asm:noise}}
\end{align}
\usetagform{default}%

\item
Finally, for the term \eqref{eq:err-tot-signal}, we have:
\begin{align}
\label{eq:err-tot-term1}
\exof{\norm{\curr[\signal]}^{2} \one_{\prev[\eventalt]}}
	&\leq \exof{\norm{\curr[\signal]}^{2}}
	\leq 2 \exof{\norm{\nabla\obj(\curr)}^{2} + \norm{\curr[\noise]}^{2}}
	\leq 2(\gbound^{2} + \noisevar),
\end{align}
with the last step following from \cref{asm:reg,asm:noise}.
\end{enumerate}
Thus, putting together all of the above, we obtain:
\begin{equation}
\exof{(\curr[R] - \prev[R]) \one_{\prev[\eventalt]}}
	\leq \bracks{\gbound^{2} + (1+\radius_{\nhd}^{2})\noisevar} \curr[\step]^{2}
\end{equation}

Going back to \eqref{eq:err-tot-cond1}, we have $\prev[R] > \eps$ if $\prev[\tilde\eventalt]$ occurs, so the last term becomes
\begin{equation}
\label{eq:err-tot-term2}
\exof{\prev[R] \one_{\prev[\tilde\eventalt]}}
	\geq \eps \exof{\one_{\prev[\tilde\eventalt]}}
	= \eps \probof{\prev[\tilde\eventalt]}.
\end{equation}
Our claim then follows by combining \cref{eq:err-tot-cond1,eq:err-tot-term1,eq:err-tot-term2}.
\end{proof}

\para{Controlling the probability of escape}

\Cref{lem:events} is the technical key to show that $\curr$ remains close to $\sol$ with high probability;
we formalize this in a final intermediate result below.

\begin{proposition}
\label{prop:contain}
Fix some tolerance level $\delta>0$.
If \cref{asm:noise} holds and \eqref{eq:SGD} is run with a step-size schedule of the form $\curr[\step] = \step/(\run + \offset)^{\pexp}$ for some sufficiently large $\offset > 0$, we have
\begin{equation}
\label{eq:contain}
\probof{\curr[\eventalt]}
	\geq 1-\delta
	\quad
	\text{for all $\run=\running$}
\end{equation}
\end{proposition}

\begin{proof}
We begin by bounding the probability of the ``large noise'' event $\curr[\tilde\eventalt] = \prev[\eventalt] \setminus \curr[\eventalt]$ as follows:
\begin{align}
\label{eq:prob-large1}
\probof{\curr[\tilde\eventalt]}
	&= \probof{\prev[\eventalt] \setminus \curr[\eventalt]}
	= \probof{\prev[\eventalt] \cap \{\curr[R] > \eps\}}
	\notag\\
	&= \exof{\one_{\prev[\eventalt]} \times \oneof{\curr[R] > \eps}}
	\notag\\
	&\leq \exof{\one_{\prev[\eventalt]} \times (\curr[R] / \eps)}
	\notag\\
	&= \exof{\curr[\tilde R]} / \eps
\end{align}
where, in the second-to-last line, we used the fact that $\curr[R] \geq 0$ (so $\oneof{\curr[R]>\eps} \leq \curr[R]/\eps$).
Now, by telescoping \eqref{eq:err-tot-cond}, we get
\begin{equation}
\label{eq:prob-large2}
\exof{\curr[\tilde R]}
	\leq \exof{\tilde R_{0}}
		+ \rconst \sum_{\runalt=\start}^{\run} \iter[\step]^{2}
		- \eps \sum_{\runalt=\start}^{\run} \probof{\preiter[\tilde\eventalt]}
\end{equation}
where we set $\rconst = \gbound^{2} + (1+\radius_{\nhd}^{2})\noisevar$.
Hence, combining \eqref{eq:prob-large1} and \eqref{eq:prob-large2}, we obtain the estimate
\begin{equation}
\sum_{\runalt=\start}^{\run} \probof{\iter[\tilde\eventalt]}
	\leq \frac{\rconst}{\eps} \sum_{\runalt=\start}^{\run} \iter[\step]^{2}
	\leq \frac{\rconst\Gamma}{\eps},
\end{equation}
where
we set $\Gamma = \sum_{\run=\start}^{\infty} \curr[\step]^{2} = \step^{2} \sum_{\run=\start}^{\infty} (\run+\offset)^{-2\pexp}$
and we used the fact that $\tilde R_{0} = 0$ and $\tilde\eventalt_{0} = \varnothing$ (by convention).

By choosing $\offset$ sufficiently large, we can ensure that $\rconst\Gamma/\eps < \delta$;
moreover, since the events $\iter[\tilde\eventalt]$ are disjoint for all $\runalt=\running$, we get
\begin{equation}
\probof*{\union_{\runalt=\start}^{\run} \iter[\tilde\eventalt]}
	= \sum_{\runalt=\start}^{\run} \probof{\iter[\tilde\eventalt]}
	\leq \delta
\end{equation}
and hence:
\begin{equation}
\probof{\curr[\eventalt]}
	= \probof*{\intersect_{\runalt=\start}^{\run} \comp{\iter[\tilde\eventalt]}}
	\geq 1 - \delta,
\end{equation}
as claimed.
\end{proof}

\para{Putting everything together}

We are finally in a position to combine all of the ingredients for the proof of \cref{thm:rate}.

\begin{proof}[Proof of \cref{thm:rate}]
To begin, define $\nhd$ and $\nhd_{1}$ as in \cref{lem:events}.
Then, by construction, we have:
\begin{equation}
\event_{\nhd}
	\equiv \{\curr\in\nhd\;\text{for all $\run=\running$}\}
	= \intersect_{\run=\start}^{\infty} \curr[\event].
\end{equation}
Since the sequence $\curr[\event]$ is decreasing and $\curr[\event] \supseteq \prev[\eventalt]$ (by the second part of \cref{lem:events}), \cref{prop:contain} yields
\begin{equation}
\probof{\event_{\nhd}}
	= \inf\nolimits_{\run} \probof{\curr[\event]}
	\geq \inf\nolimits_{\run} \probof{\prev[\eventalt]}
	\geq 1-\delta,
\end{equation}
provided that $\offset$ is chosen large enough.
This proves the first part of the theorem, \ie to the effect that $\curr$ remains close to $\sol$ with probability at least $1-\delta$.

For the second part of the theorem, \cref{prop:dist-loc} readily gives
\begin{equation}
\next[\breg] \one_{\curr[\event]}
	\leq (1-2\qmin\curr[\step]) \curr[\breg] \one_{\curr[\event]}
		+ \bracks{\curr[\step] \curr[\snoise] + \tfrac{1}{2} \curr[\step]^{2} \norm{\curr[\signal]}^{2}} \one_{\curr[\event]}.
\end{equation}
Now, for any given $\step$, we can choose $\offset$ sufficiently large so that $\inf_{\run}(1-2\qmin\curr[\step]) > 0$ and $\probof{\event_{\nhd}} \geq 1-\delta$ (the latter by \cref{prop:contain}).
Moreover, working as in the proof of \cref{lem:events}, we get
\begin{align}
\exof*{\parens{\curr[\step] \curr[\snoise] + \tfrac{1}{2} \norm{\curr[\signal]}^{2}} \one_{\curr[\event]}}
	&= \exof*{\one_{\curr[\event]}
		\exof*{\curr[\step] \curr[\snoise] + \tfrac{1}{2} \curr[\step]^{2} \norm{\curr[\signal]}^{2} \given \curr[\filter]}}
	\notag\\
	&\leq \curr[\step]^{2} (\gbound^{2} + \noisevar).
\end{align}
Then, letting $\bar\breg_{\run} = \exof{\curr[\breg] \one_{\curr[\event]}} \geq 0$ and recalling that $\next[\event] \subseteq \curr[\event]$ (so $\one_{\next[\event]} \leq \one_{\curr[\event]}$), the two estimates above yield
\begin{align}
\bar\breg_{\run+1}
	\leq \exof{\next[\breg] \one_{\curr[\event]}}
	&\leq (1-2\qmin\curr[\step]) \bar\breg_{\run}
		+ (\gbound^{2} + \noisevar) \curr[\step]^{2}
	\notag\\
	&\leq \bracks*{1 - \frac{2\qmin\step}{(\run+\offset)^{\pexp}}} \bar\breg_{\run}
		+ \frac{(\gbound^{2} + \noisevar)\step^{2}}{(\run+\offset)^{2\pexp}}.
\end{align}
Thus, by \cref{lem:Chung}, we obtain the bounds:
\begin{subequations}
\begin{alignat}{2}
\label{eq:Chung-small}
\bar\breg_{\run}
	&\leq \frac{\gbound^{2} + \noisevar}{2\qmin} \frac{\step}{\run^{\pexp}}
		+ o\parens*{\frac{1}{\run^{\pexp}}}
		&\qquad
		&\text{if $\pexp<1$},
\shortintertext{and}
\label{eq:Chung-large}
\bar\breg_{\run}
	&\leq \frac{\gbound^{2} + \noisevar}{2\qmin\step - 1} \frac{\step^{2}}{\run}
		+ o\parens*{\frac{1}{\run^{\pexp}}}
		&\qquad
		&\text{if $\pexp=1$},
\end{alignat}
\end{subequations}
provided that $2\qmin\step > 1$ for the latter.
The claim of the theorem then follows by noting that
\begin{equation}
\exof{\norm{\curr - \sol}^{2} \given \event_{\nhd}}
	\leq \frac{\exof{\norm{\curr - \sol}^{2} \one_{\event_{\nhd}}}}{\probof{\event_{\nhd}}}
	\leq \frac{2}{1-\delta} \bar\breg_{\run}
\end{equation}
and applying \cref{eq:Chung-small,eq:Chung-large}.
\end{proof}


\begin{figure*}[tbp]
\centering
\footnotesize
\begin{subfigure}{.45\linewidth}
\centering
\includegraphics[height=5cm]{Figures/train_loss_cutoff_math.pdf}%
\caption{Training loss}
\label{fig:ResNet-100-train-loss}
\end{subfigure}
\quad
\begin{subfigure}{.45\linewidth}
\centering
\includegraphics[height=5cm]{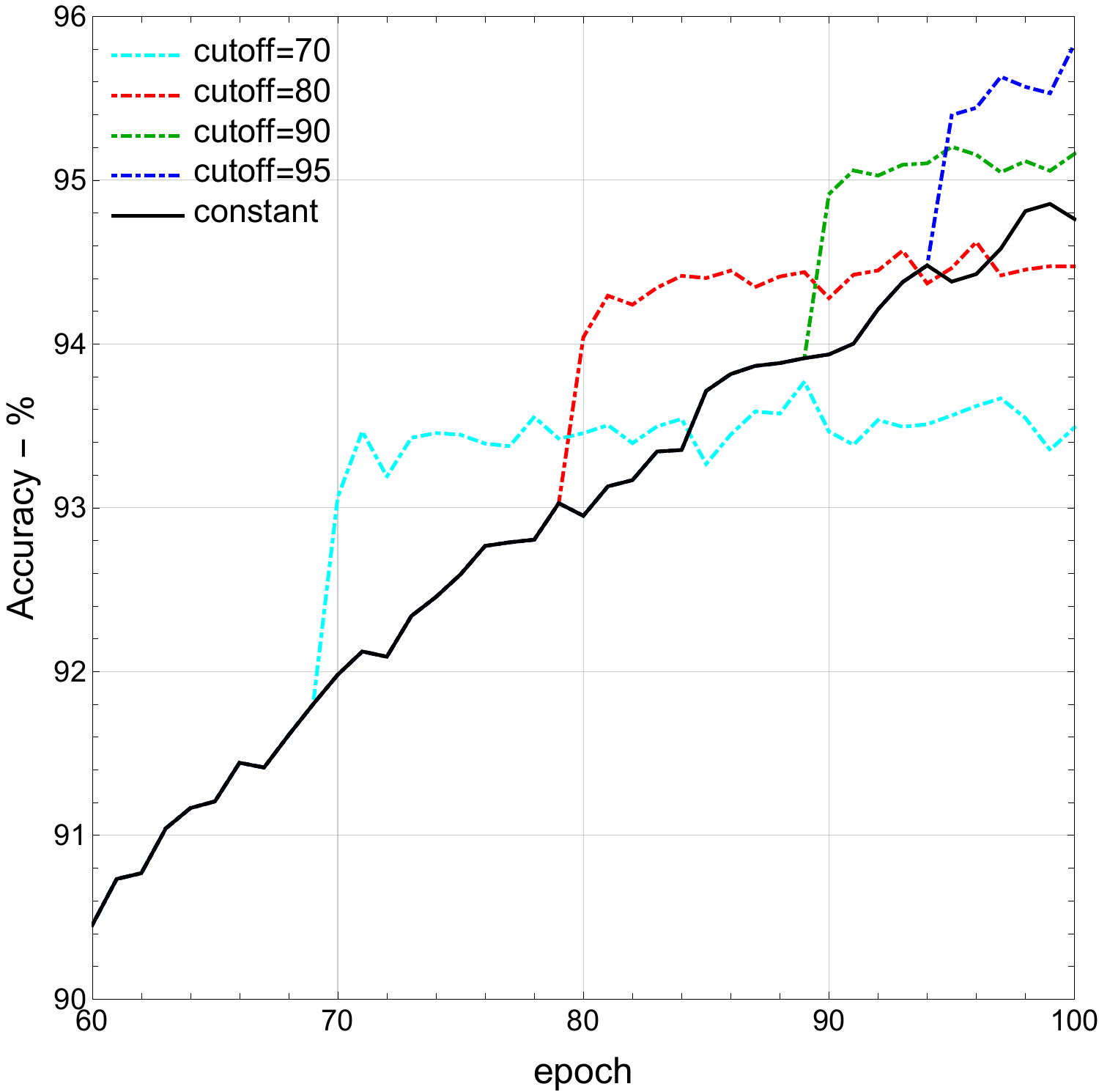}%
\caption{Training accuracy}
\label{fig:ResNet-100-train-acc}
\end{subfigure}
\begin{subfigure}{.45\linewidth}
\centering
\includegraphics[height=5cm]{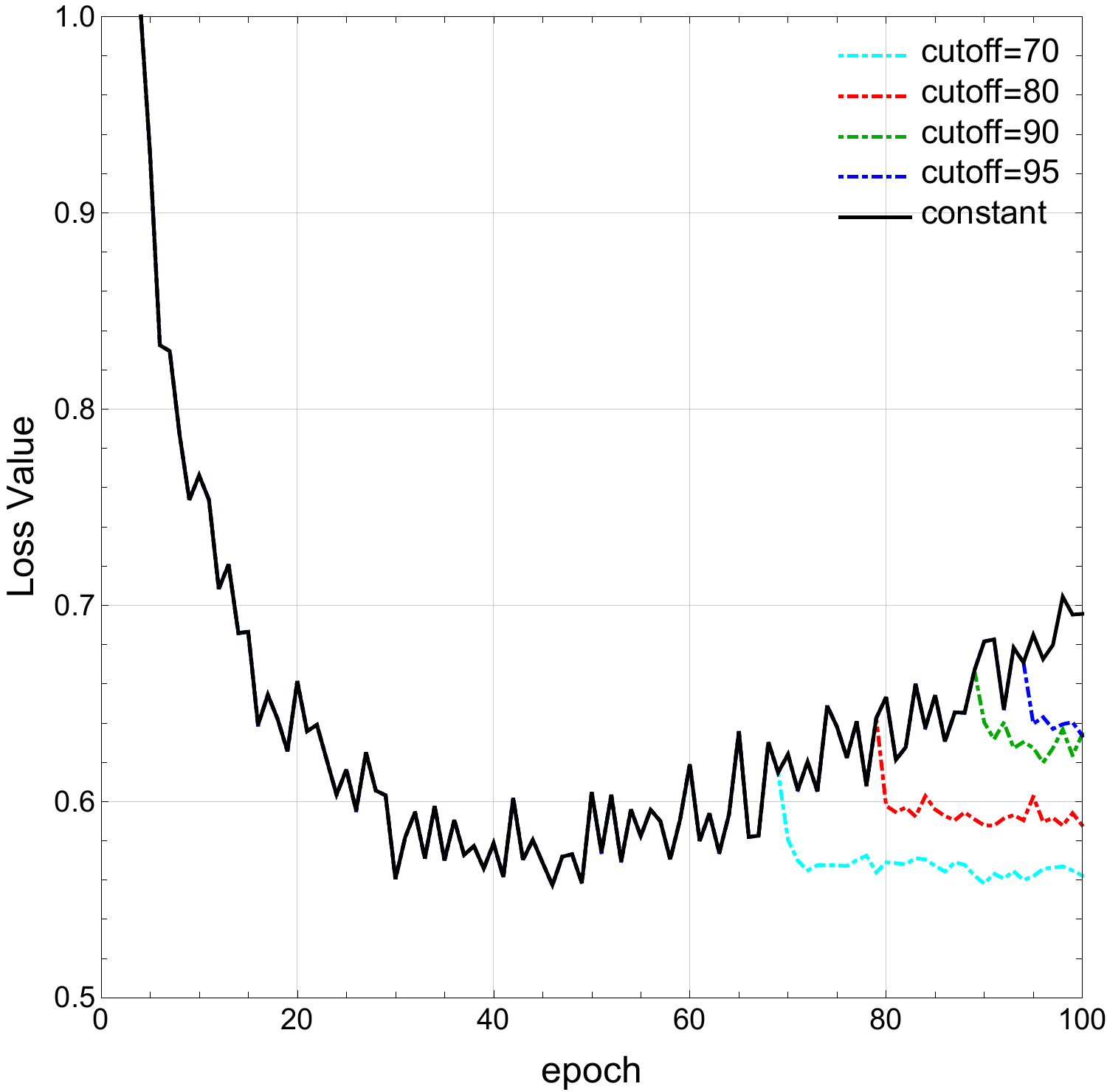}%
\caption{Test loss}
\label{fig:ResNet-100-test-loss}
\end{subfigure}
\quad
\begin{subfigure}{.45\linewidth}
\centering
\includegraphics[height=5cm]{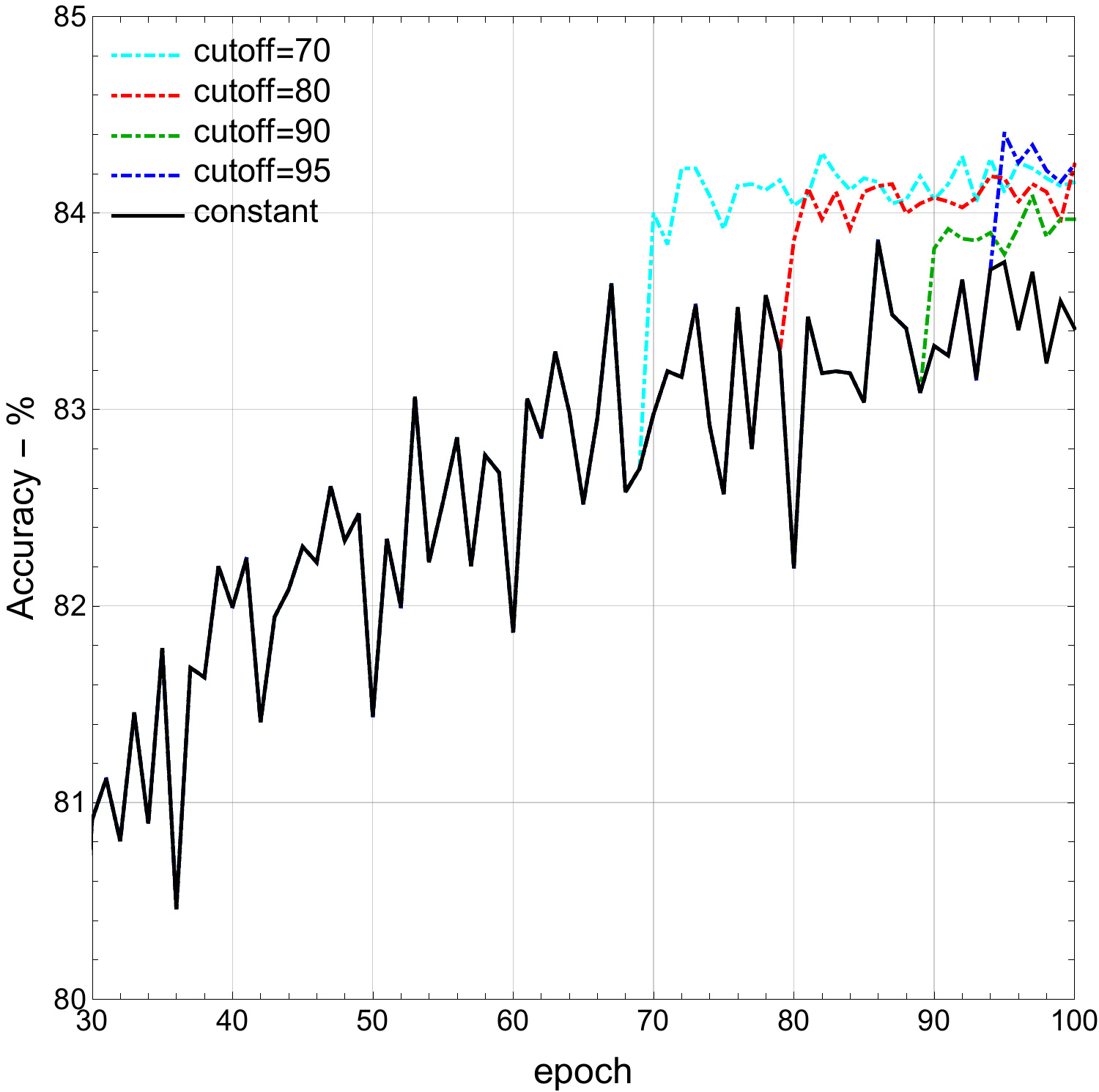}%
\caption{Test accuracy}
\label{fig:ResNet-100-test-acc}
\end{subfigure}
\caption{Results for training ResNet18 model for classification over CIFAR10 dataset, with cooldown heuristic. Constant step-size SGD is run for 100 epochs and cooldown phase starts at epochs 70, 80, 90 and 95, with diminishing step-size policy of $1/n$.}
\label{fig:ResNet-app}
\end{figure*}


\section{Numerical experiments}
\label{app:numerics}

In this appendix, we present some more details on our ResNet training setup and some additional numerical results.
We used the \emph{python/pytorch} implementation of Resnet18 from the \emph{torchvision} package and, for consistency, we downloaded the CIFAR10 dataset from the same package.
For training/evaluation purposes, we used the the standard training/test split of 50000/10000 examples, with training and test batches of size 120.

The purposes of our experiments is to demonstrate the possible benefits of the ``cooldown'' heuristic that is derived from our convergence analysis in \cref{sec:analysis}.
To that end, we initially trained the model with constant step-size \eqref{eq:SGD} whose step-size is picked through grid-search over the set $\{1, 10^{-1}, 10^{-2}, 10^{-3}, 10^{-4}\}$.
We then took checkpoints of the model at certain epochs and launched the cooldown heuristic from such points with a $1/n$ step-size policy.
It is important to emphasize that the iteration counter $n$ starts at the first iteration of cooldown phase so that we have a ``continuous'' sequence of step-sizes across epochs.
In \cref{fig:ResNet-app}, we provide the complementary plots for the setting described in Section \ref{sec:numerics}, again exhibiting a clear benefit (especially in test accuracy and loss) when using the cool-down heuristic for the last part of the experiment's runtime budget.

\bibliographystyle{plainnat}
\bibliography{IEEEabrv,bibtex/Bibliography}

\end{document}